\documentclass[11pt]{amsart}  
\usepackage[hmargin={2cm,2cm},vmargin={2cm,2cm}]{geometry} 

\usepackage{url}
\usepackage{stmaryrd}
\usepackage[square]{natbib} 
\usepackage[leqno]{amsmath}
\usepackage{amssymb,amsthm}
\usepackage{bbm}
\usepackage{dsfont}
\usepackage[all,arc]{xy}
\DeclareMathAlphabet{\mathpzc}{OT1}{pzc}{m}{it}

\theoremstyle{plain}
\newtheorem{theorem}{Theorem}[section]
\newtheorem{lemma}[theorem]{Lemma}
\newtheorem{proposition}[theorem]{Proposition}
\newtheorem{corollary}[theorem]{Corollary}

\theoremstyle{definition}

\newtheorem{example}[theorem]{Example}

\theoremstyle{remark}
\newtheorem{remark}[theorem]{Remark}

\newenvironment{eqcond}{\begin{enumerate}\renewcommand{\theenumi}{\roman{enumi}}}{\end{enumerate}}
\renewcommand{\theenumi}{\arabic{enumi}}

\usepackage{chngcntr}
\counterwithin*{equation}{section}

\newcommand{\Rw}{\Rightarrow}

\newcommand{\hrw}{\hookrightarrow}

\newcommand{\RLw}{\Leftrightarrow}

\newcommand{\fra}{\mathfrak{a}}

\newcommand{\frf}{\mathfrak{f}}

\newcommand{\frp}{\mathfrak{p}}

\newcommand{\fru}{\mathfrak{u}}
\newcommand{\frw}{\mathfrak{w}}
\newcommand{\frv}{\mathfrak{v}}
\newcommand{\frx}{\mathfrak{x}}
\newcommand{\fry}{\mathfrak{y}}

\newcommand{\calA}{\mathcal{A}}

\newcommand{\calO}{\mathcal{O}}
\newcommand{\calS}{\mathcal{S}}

\newcommand{\frX}{\mathfrak{X}}

\DeclareMathOperator{\yoneda}{\mathpzc{y}}
\DeclareMathOperator{\yonedaOP}{\mathpzc{h}}
\DeclareMathOperator{\yonedaT}{\mathpzc{Y}}
\DeclareMathOperator{\yonmult}{\mathpzc{m}}
\DeclareMathOperator{\formalball}{\mathpzc{B}}
\DeclareMathOperator{\fammod}{\mathpzc{F}}
\DeclareMathOperator{\Sup}{Sup}

\DeclareMathOperator{\Bimorph}{Bimorph}
\DeclareMathOperator{\ev}{ev}

\DeclareMathOperator{\comp}{comp}

\newcommand{\mate}[1]{\,^\ulcorner\! #1^\urcorner}

\newcommand{\fspstr}[2]{\llbracket #1,#2\rrbracket}

\newcommand{\catfont}[1]{\mathsf{#1}}

\newcommand{\SET}{\catfont{Set}}

\newcommand{\ORD}{\catfont{Ord}}
\newcommand{\MET}{\catfont{Met}}
\newcommand{\COCTSAP}{\catfont{InjApp}_\mathrm{sup}}
\newcommand{\TOP}{\catfont{Top}}
\newcommand{\PSTOP}{\catfont{PsTop}}
\newcommand{\AP}{\catfont{App}}
\newcommand{\PSAP}{\catfont{PsApp}}
\newcommand{\SUP}{\catfont{Sup}}
\newcommand{\CONTLAT}{\catfont{ContLat}}
\newcommand{\OPCONTLAT}{\catfont{ContLat_*}}
\newcommand{\ORDCH}{\catfont{OrdCompHaus}}

\newcommand{\METCH}{\catfont{MetCompHaus}}

\newcommand{\two}{\catfont{2}}

\newcommand{\kto}{\relbar\joinrel\rightharpoonup}

\newcommand{\kmodto}{\,{\kto\hspace*{-2.8ex}{\circ}\hspace*{1.3ex}}}

\newcommand{\monadfont}[1]{\mathbbm{#1}}

\newcommand{\mU}{\monadfont{U}}
\newcommand{\mP}{\monadfont{P}}

\newcommand{\umonad}{(U,e,m)}
\newcommand{\pmonad}{(P,\yoneda,\yonmult)}

\newcommand{\doo}[1]{\overset{\centerdot}{#1}}
\newcommand{\eps}{\varepsilon}
\newcommand{\op}{\mathrm{op}}
\newcommand{\sep}{\mathrm{sep}}

\newcommand{\true}{\mathsf{true}}
\newcommand{\false}{\mathsf{false}}
\newcommand{\trunminus}{\ominus}

\newcommand{\field}[1]{\mathds{#1}}

\newcommand{\N}{\field{N}}

\usepackage[hypertexnames=false]{hyperref} 

\begin{document}
\newcounter{counter}
\title{Approaching metric domains}

\author{Gon\c{c}alo Gutierres}
\address{CMUC, Department of Mathematics, University of Coimbra, 3001-454
Coimbra, Portugal}
\email{ggutc@mat.uc.pt}

\author{Dirk Hofmann}
\address{CIDMA, Department of Mathematics, University of Aveiro, 3810-193 Aveiro, Portugal}
\email{dirk@ua.pt}

\thanks{Partial financial assistance by Centro de Matem\'{a}tica da Universidade de Coimbra/FCT, Centro de Investiga\c{c}\~ao e Desenvolvimento em Matem\'atica e Aplica\c{c}\~oes da Universidade de Aveiro/FCT and the project MONDRIAN (under the contract PTDC/EIA-CCO/108302/2008) is gratefully acknowledged.}


\begin{abstract}
In analogy to the situation for continuous lattices which were introduced by Dana Scott as precisely the injective T$_0$ spaces via the (nowadays called) Scott topology, we study those metric spaces which correspond to injective T$_0$ approach spaces and characterise them as precisely the continuous lattices equipped with an unitary and associative $[0,\infty]$-action. This result is achieved by a thorough analysis of the notion of cocompleteness for approach spaces.
\end{abstract}


\maketitle

\section*{Introduction}

Domain theory is generally concerned with the study of \emph{ordered sets} admitting certain (typically up-directed) suprema and a notion of approximation, here the latter amounts to saying that each element is a (up-directed) supremum of suitably defined ``finite'' elements. From a different perspective, domains can be viewed as very particular \emph{topological spaces}; in fact, in his pioneering paper \citep{Sco_ContLat} Scott introduced the notion of continuous lattice precisely as injective topological T$_0$ space. Yet another point of view was added in \citep{Day_Filter,Wyl85} where continuous lattices are shown to be precisely the \emph{algebras} for the filter monad. Furthermore, suitable submonads of the filter monad have other types of domains as algebras (for instance, continuous Scott domains \citep{Book_ContLat} or Nachbin's ordered compact Hausdorff spaces \citep{Nach_TopOrd}), and, as for continuous lattices, these domains can be equally seen as objects of topology and of order theory. This interplay between topology and algebra is very nicely explained in \citep{EF_SemDom} where, employing a particular property of monads of filters, the authors obtain ``new proofs and [\ldots] new characterizations of semantic domains and topological spaces by injectivity''.

Since \citeauthor{Law_MetLogClo}'s ground-breaking paper \citep{Law_MetLogClo} it is known that an individual metric spaces $X$ can be viewed as a category with objects the points of $X$, and the distance
\[
d(x,y)\in[0,\infty]
\]
plays the role of the ``hom-set'' of $x$ and $y$. More modest, one can think of a metric $d:X\times X\to[0,\infty]$ as an order relation on $X$ with truth-values in $[0,\infty]$ rather than in the Boolean algebra $\two=\{\false,\true\}$. In fact, writing $0$ instead of $\true$, $\geqslant$ instead of $\Rw$ and additon $+$ instead of and $\&$, the reflexivity and transitivity laws of an ordered set become
\begin{align*}
0\ge d(x,x) &&\text{and} && d(x,y)+d(y,z)\ge d(x,z) &&(x,y,x\in X),
\end{align*}
and in this paper we follow Lawvere's point of view and assume no further properties of $d$. As pointed out in \citep{Law_MetLogClo}, ``this connection is more fruitful than a mere analogy, because it provides a sequence of mathematical theorems, so that enriched category theory can suggest new directions of research in metric space theory and conversely''. A striking example of commonality between category (resp.\ order) theory and metric theory was already given in \citep{Law_MetLogClo} where it is shown that Cauchy sequences correspond to adjoint (bi)modules and convergence of Cauchy sequences corresponds to representabilty of these modules. Eventually, this amounts to saying that a metric space is Cauchy complete if and only if it admits ``suprema'' of certain ``down-sets'' ($=$ morphisms of type $X^\op\to[0,\infty]$), here ``suprema'' has to be taken in the sense of weighted colimit of enriched category theory \citep{EK_CloCat,Kel_EnrCat}. Other types of ``down-sets'' $X^\op\to[0,\infty]$ specify other properties of metric spaces: forward Cauchy sequences (or nets) (see \citep{BBR_GenMet}) can be represented by so called flat modules (see \citep{Vic_LocComplI}) and their limit points as ``suprema'' of these ``down-sets'', and the formal ball model of a metric space relates to its cocompletion with respect to yet another type of ``down-sets'' (see \citep{Rut98a,KW_FormBall}).

The particular concern of this paper is to contribute to the development of metric domain theory. Due to the many facets of domains, this can be pursued by either
\begin{enumerate}
\item formulation order-theoretic concepts in the logic of $[0,\infty]$, 
\item considering injective ``$[0,\infty]$-enriched topological spaces'', or
\item studying the algebras of ``metric filter monads''. 
\end{enumerate}
Inspired by \citep{Law_MetLogClo}, there is a rich literature employing the first point of view, including \citeauthor{Wag_PhD}'s Ph.D.\ thesis \citep{Wag_PhD}, the work of the Amsterdam research group at CWI \citep{BBR_GenMet,Rut98a}, the work of \citeauthor{FSW_QDT} on continuity spaces \citep{Kop88,Flagg_ComplCont,Flagg_QuantCont,FSW_QDT,FK_ContSp}, and the work of \citeauthor{Was_DomGirQant} with various coauthors on approximation and the formal ball model \citep{Was_DomGirQant,HW_AppVCat} and \citep{KW_FormBall}. However, in this paper we take a different approach and concentrate on the second and third aspect above. Our aim is to connect the theory of metric spaces with the theory of ``$[0,\infty]$-enriched topological spaces'' in a similar fashion as domain theory is supported by topology, where by ``$[0,\infty]$-enriched topological spaces'' we understand \citeauthor{Low_ApBook}'s approach spaces \citep{Low_ApBook}. (In a nutshell, an approach space is to a topological space what a metric space is to an ordered set: it can be defined in terms of ultrafilter convergence where one associates to an ultrafilter $\frx$ and a point $x$ a value of convergence $a(\frx,x)\in[0,\infty]$ rather then just saying that $\frx$ converges to $x$ or not.) This idea was already pursued in \citep{Hof_Cocompl} and \citep{Hof_DualityDistSp} were among others it is shown that
\begin{itemize}
\item injective T$_0$ approach spaces correspond bijectively to a class of metric spaces, henceforth thought of as ``continuous metric spaces'',
\item these ``continuous metric spaces'' are precisely the algebras for a certain monad on $\SET$, henceforth thought of as the ``metric filter monad'',
\item the category of injective approach spaces and approach maps is Cartesian closed.
\end{itemize}
Here we continue this path and
\begin{itemize}
\item recall the theory of metric and approach spaces as generalised orders (resp.\ categories),
\item characterise metric compact Hausdorff spaces as the (suitably defined) stably compact approach spaces, and
\item show that injective T$_0$ approach spaces (aka ``continuous metric spaces'') can be equivalently described as continuous lattices equipped with an unitary and associative action of the continuous lattice $[0,\infty]$. This result is achieved by a thorough analysis of the notion of cocompleteness for approach spaces.
\end{itemize}

\subsubsection*{Warning}

The underlying order of a topological space $X=(X,\calO)$ we define as
\[
 x\le y\quad\text{whenever}\quad \doo{x}\to y
\]
which is equivalent to $\calO(y)\subseteq\calO(x)$, hence it is the \emph{dual of the specialisation order}. As a consequence, the underlying order of an injective T$_0$ topological space is the dual of a continuous lattice; and our results are stated in terms of these op-continuous lattices. We hope this does not create confusion. 

\section{Metric spaces}

\subsection{Preliminaries}\label{Subsect:Prelim}

According to the Introduction, in this paper we consider \emph{metric spaces} in a more general sense: a metric $d:X\times X\to[0,\infty]$ on set $X$ is only required to satisfy 
\begin{align*}
0\geqslant d(x,x) &&\text{and}&& d(x,y)+d(y,z)\geqslant d(x,z).
\end{align*}
For convenience we often also assume $d$ to be \emph{separated} meaning that $d(x,y)=0=d(y,x)$ implies $x=y$ for all $x,y\in X$. With this nomenclature, ``classical'' metric spaces appear now as separated, symmetric ($d(x,y)=d(y,x)$) and finitary ($d(x,y)<\infty$) metric spaces. A map $f:X\to X'$ between metric spaces $X=(X,d)$ and $X'=(X',d')$ is a \emph{metric map} whenever $d(x,y)\geqslant d'(f(x),f(y))$ for all $x,y\in X$. The category of metric spaces and metric maps we denote as $\MET$. To every metric space $X=(X,d)$ one associates its \emph{dual space} $X^\op=(X,d^\circ)$ where $d^\circ(x,y)=d(y,x)$, for all $x,y\in X$. Certainly, the metric $d$ on $X$ is symmetric if and only if $X=X^\op$. Every metric map $f$ between metric spaces $X$ and $Y$ is also a metric  map of type $X^\op\to Y^\op$, hence taking duals is actually a functor $(-)^\op:\MET\to\MET$ which sends $f:X\to Y$ to $f^\op:X^\op\to Y^\op$.

There is a canonical forgetful functor $(-)_p:\MET\to\ORD$: for a metric space $(X,d)$, put
\[
x\le y \text{ whenever } 0\geqslant d(x,y),
\]
and every metric map preserves this order. Also note that $(-)_p:\MET\to\ORD$ has a left adjoint $\ORD\to\MET$ which interprets an order relation $\le$ on $X$ as the metric
\[
 d(x,y)=
\begin{cases}
 0 & \text{if }x\le y,\\
 \infty & \text{else.}
\end{cases}
\]
In particular, if $X$ is a discrete ordered set meaning that the order relation is just the equality relation on $X$, then one obtains the discrete metric on $X$ where $d(x,x)=0$ and $d(x,y)=\infty$ for $x\neq y$.

The induced order of a metric space extends point-wise to metric  maps making $\MET$ an \emph{ordered category}, which enables us to talk about \emph{adjunctions}. Here metric  maps $f:(X,d)\to (X',d')$ and $g:(X',d')\to(X,d)$ form an adjunction, written as $f\dashv g$, if $1_X\le g\cdot f$ and $f\cdot g\le 1_{X'}$. Equivalently, $f\dashv g$ if and only if, for all $x\in X$ and $x'\in X'$,
\[
 d'(f(x),x')=d(x,g(x')).
\]
The formula above explains the costume to call $f$ left adjoint and $g$ right adjoint. We also recall that adjoint maps determine each other meaning that $f\dashv g$ and $f\dashv g'$ imply $g\simeq g'$, and $f\dashv g$ and $f'\dashv g$ imply $f\simeq f'$.

The category $\MET$ is complete and, for instance, the product $X\times Y$ of metric spaces $X=(X,a)$ and $(Y,b)$ is given by the Cartesian product of the sets $X$ and $Y$ equipped with the $\max$-metric
\[
d((x,y),(x',y'))=\max(a(x,x'),b(y,y')).
\]
More interestingly to us is the plus-metric
\[
d'((x,y),(x',y'))=a(x,x')+b(y,y')
\]
on the set $X\times Y$, we write $a\oplus b$ for this metric and denote the resulting metric space as $X\oplus Y$. Note that the underlying order of $X\oplus Y$ is just the product order of $X_p$ and $Y_p$ in $\ORD$. Furthermore, for metric  maps $f:X\to Y$ and $g:X'\to Y'$, the product of $f$ and $g$ gives a metric  map $f\oplus g:X\oplus X'\to Y\oplus Y'$, and we can view $\oplus$ as a functor $\oplus:\MET\times\MET\to\MET$. This operation is better behaved then the product $\times$ in the sense that, for every metric space $X$, the functor $X\oplus-:\MET\to\MET$ has a right adjoint $(-)^X:\MET\to\MET$ sending a metric space $Y=(Y,b)$ to
\begin{align*}
Y^X&=\{h:X\to Y\mid h\text{ in $\MET$}\} &&\text{with distance}\qquad [h,k]=\sup_{x\in X}b(h(x),k(x)),
\end{align*}
and a metric  map $f:Y_1\to Y_2$ to
\[
f^X:Y_1^X\to Y_2^X,\,h\mapsto f\cdot h.
\]
In particular, if $X$ is a discrete space, then $Y^X$ is just the $X$-fold power of $Y$.

In the sequel we will pay particular attention to the metric space $[0,\infty]$, with metric $\mu$ defined by
\[\mu(u,v)=v\trunminus u:=\max\{v-u,0\},\]
for all $u,v\in[0,\infty]$. Then the underlying order on $[0,\infty]$ is the ``greater or equal relation'' $\geqslant$. The ``turning around'' of the natural order of $[0,\infty]$ might look unmotivated at first sight but has its roots in the translation of ``$\false\le\true$'' in $\two$ to ``$\infty\geqslant 0$'' in $[0,\infty]$. We also note that $u+-:[0,\infty]\to[0,\infty]$ is left adjoint to $\mu(u,-):[0,\infty]\to[0,\infty]$ with respect to $\geqslant$ in $[0,\infty]$. However, in the sequel we will usually refer to the natural order $\leqslant$ on $[0,\infty]$ with the effect that some formulas are dual to what one might expect. For instance, the underlying monotone map of a metric  map of type $X^\op\to[0,\infty]$ is of type $X_p\to[0,\infty]$; and when we talk about a supremum ``$\bigvee$'' in the underlying order of a generic metric space, it specialises to taking infimum ``$\inf$'' with respect to the usual order $\leqslant$ on $[0,\infty]$.

For every set $I$, the maps
\begin{align*}
\inf:[0,\infty]^I&\to[0,\infty] &&\text{and}& \sup:[0,\infty]^I &\to[0,\infty]\\
\varphi &\mapsto \inf_{i\in I}\varphi(i) &&& \varphi &\mapsto \sup_{i\in I}\varphi(i)
\intertext{are metric  maps, and so are}
+:[0,\infty]\oplus [0,\infty]&\to[0,\infty] &&\text{and}&
\mu:[0,\infty]^\op\oplus [0,\infty]&\to[0,\infty].\\
 (u,v) &\mapsto u+v &&& (u,v) &\mapsto v\trunminus u
\end{align*}
More general, for a metric space $X=(X,d)$, the metric $d$ is a metric  map $d:X^\op\oplus X\to [0,\infty]$. Its mate is the \emph{Yoneda embedding}
\[
 \yoneda_X:=\mate{d}:X\to[0,\infty]^{X^\op},\,x\mapsto d(-,x),
\]
which satisfies indeed $d(x,y)=[\yoneda_X(x),\yoneda_X(y)]$ for all $x,y\in X$ thanks to the Yoneda lemma which states that
\[
[\yoneda_X(x),\psi]=\psi(x),
\]
for all $x\in X$ and $\psi\in[0,\infty]^{X^\op}$.

\subsection{Cocomplete metric spaces}\label{SubSect:CocomplMetSp}

In this subsection we have a look at metric spaces ``through the eyes of category (resp.\ order) theory'' and study the existence of suprema of ``down-sets'' in a metric space. This is a particular case of the notion of weighted colimit of enriched categories (see \citep{EK_CloCat,Kel_EnrCat,KS_Colim}, for instance), and in this and the next subsection we spell out the meaning for metric spaces of general notions and results of enriched category theory. 

For a metric space $X=(X,d)$ and a ``down-set'' $\psi:X^\op\to[0,\infty]$ in $\MET$, an element $x_0\in X$ is a \emph{supremum} of $\psi$ whenever, for all $x\in X$,
\begin{equation}\label{Eq:SupremumMet}
d(x_0,x)=\sup_{y\in X}(d(y,x)\trunminus\psi(y)).
\end{equation}
Suprema are unique up to equivalence $\simeq$, we write $x_0\simeq\Sup_X(\psi)$ and will frequently say \emph{the} supremum.  Furthermore, a metric  map $f:(X,d)\to (X',d')$ preserves the supremum of $\psi\in[0,\infty]^{X^\op}$ whenever
\[
 d'(f(\Sup_X(\psi)),x')=\sup_{x\in X}(d'(f(x),x')\trunminus\psi(x))
\]
for all $x'\in X'$. As for ordered sets:
\begin{lemma}
Left adjoint metric  maps preserve all suprema.
\end{lemma}

A metric space $X=(X,d)$ is called \emph{cocomplete} if every ``down-set'' $\psi:X^\op\to[0,\infty]$ has a supremum. This is the case precisely if, for all $\psi\in[0,\infty]^{X^\op}$ and all $x\in X$,
\begin{equation*}
d(\Sup_X(\psi),x)=\sup_{y\in X}(d(y,x)\trunminus\psi(y))=[\psi,\yoneda_X(x)];
\end{equation*}
hence $X$ is cocomplete if and only if the Yoneda embedding $\yoneda_X:X\to[0,\infty]^{X^\op}$ has a left adjoint $\Sup_X:[0,\infty]^{X^\op}\to X$ in $\MET$. More generally, one has (see \citep{Hof_Cocompl}, for instance)
\begin{proposition}
For a metric space $X$, the following conditions are equivalent.
\begin{eqcond}
\item $X$ is injective (with respect to isometries).
\item $\yoneda_X:X\to[0,\infty]^{X^\op}$ has a left inverse.
\item $\yoneda_X$ has a left adjoint.
\item $X$ is cocomplete.
\end{eqcond}
\end{proposition}
Here a metric map $i:(A,d)\to(B,d')$ is called \emph{isometry} if one has $d(x,y)= d'(i(x),i(y))$ for all $x,y\in A$, and $X$ is injective if, for all isometries $i:A\to B$ and all $f:A\to X$ in $\MET$, there exists a metric  map $g:B\to X$ with $g\cdot i\simeq f$. 

Dually, an infimum of an ``up-set'' $\varphi:X\to[0,\infty]$ in $X=(X,d)$ is an element $x_0\in X$ such that, for all $x\in X$,
\[
d(x,x_0)=\sup_{y\in X}(d(x,y)\trunminus\varphi(y)).
\]
A metric space $X$ is \emph{complete} if every ``up-set'' has an infimum. By definition, an infimum of $\varphi:X\to[0,\infty]$ in $X$ is a supremum of $\varphi:(X^\op)^\op\to[0,\infty]$ in $X^\op$, and everything said above can be repeated now in its dual form. In particular, with $\yonedaOP_X:X\to\left([0,\infty]^X\right)^\op,\, x\mapsto d(x,-)$ denoting the \emph{contravariant Yoneda embedding} (which is the dual of $\yoneda_{X^\op}:X^\op\to[0,\infty]^X$):
\begin{proposition}
For a metric space $X$, the following conditions are equivalent.
\begin{eqcond}
\item $X$ is injective (with respect to isometries).
\item $\yonedaOP_X:X\to\left([0,\infty]^X\right)^\op$ has a left inverse.
\item $\yonedaOP_X$ has a left adjoint.
\item $X$ is complete.
\end{eqcond}
\end{proposition}
\begin{corollary}
A metric space is complete if and only if it is cocomplete.
\end{corollary}
This latter fact can be also seen in a different way. To every ``down-set'' $\psi:X^\op\to[0,\infty]$ one assigns its ``up-set of upper bounds''
\[
\psi^+:X\to[0,\infty],\,x\mapsto\sup_{y\in X}(d(y,x)\trunminus\psi(y)),
\]
and to every ``up-set'' $\varphi:X\to[0,\infty]$ its ``down-set of lower bounds''
\[
\varphi^-:X^\op\to[0,\infty],\,x\mapsto\sup_{y\in X}(d(x,y)\trunminus\varphi(y)).
\]
This way one defines an adjunction $(-)^+\dashv (-)^-$
\[
\xymatrix{\left([0,\infty]^X\right)^\op\ar@/^1.5ex/[rr]^{(-)^-}\ar@{}[rr]|\top &&
[0,\infty]^{X^\op}\ar@/^1.5ex/[ll]^{(-)^+}\\
 & X\ar[ul]^{\yonedaOP_X}\ar[ur]_{\yoneda_X}}
\]
in $\MET$ where both maps commute with the Yoneda embeddings. Therefore a left inverse of $\yoneda_X$ produces a left inverse of $\yonedaOP_X$, and vice versa. The adjunction $(-)^+\dashv(-)^-$ is also know as the \emph{Isbell conjugation adjunction}.

For every (co)complete metric space $X=(X,d)$, its underlying ordered set $X_p$ is (co)complete as well. This follows for instance from the fact that $(-)_p:\MET\to\ORD$ preserves injective objects. Another argument goes as follows. For every (down-set) $A\subseteq X$, one defines a metric  map
\[
\psi_A:X^\op\to[0,\infty],\,x\mapsto\inf_{a\in A}d(x,a),
\]
and a supremum $x_0$ of $\psi_A$ must satisfy, for all $x\in X$,
\[
d(x_0,x)=\sup_{y\in X}(d(y,x)\trunminus\psi_A(y))
=\sup_{a\in A}\sup_{y\in X}(d(y,x)\trunminus d(y,x))
=\sup_{a\in A}d(a,x).
\]
Therefore $x_0$ is not only a supremum of $A$ in the ordered set $X_p$, it is also preserved by every monotone map $d(-,x):X_p\to[0,\infty]$. 
\begin{lemma}
Let $X=(X,d)$ be a metric space and let $x_0\in X$ and $A\subseteq X$. Then $x_0$ is the supremum of $\psi_A$ if and only if $x_0$ is the (order theoretic) supremum of $A$ and, for every $x\in X$, the monotone map $d(-,x):X_p\to[0,\infty]$ preserves this supremum.
\end{lemma}

\subsection{Tensored metric spaces}\label{Subsect:TensMet}

We are now interested in those metric spaces $X=(X,d)$ which admit suprema of ``down-sets'' of the form $\psi=d(-,x)+u$ where $x\in X$ and $u\in[0,\infty]$. In the sequel we write $x+u$ instead of $\Sup_X(\psi)$. According to \eqref{Eq:SupremumMet}, the element $x+u\in X$ is characterised up to equivalence by
\[
 d(x+u,y)=d(x,y)\trunminus u,
\]
for all $y\in X$. A metric  map $f:(X,d)\to (X',d')$ preserves the supremum of $\psi=d(-,x)+u$ if and only if $f(x+u)\simeq f(x)+u$. Dually, an infimum of an ``up-set'' of the form $\varphi=d(x,-)+u$ we denote as $x\trunminus u$, it 
is characterised up to equivalence by
\[
 d(y,x\trunminus u)=d(y,x)\trunminus u.
\]
One calls a metric space \emph{tensored} if it admits all suprema $x+u$, and \emph{cotensored} if $X$ admits all infima $x\trunminus u$.
\begin{example}
The metric space $[0,\infty]$ is tensored and cotensored where $x+u$ is given by addition and  $x\trunminus u=\max\{x-u,0\}$. 
\end{example}

Note that $X=(X,d)$ is tensored if and only if every $d(x,-):X\to[0,\infty]$ has a left adjoint $x+(-):[0,\infty]\to X$ in $\MET$, and $X$ is cotensored if and only if every $d(-,x):X^\op\to[0,\infty]$ has a left adjoint $x\trunminus(-):[0,\infty]\to X^\op$ in $\MET$. Furthermore, if $X$ is tensored and cotensored, then $(-)+u:X\to X$ is left adjoint to $(-)\trunminus u:X\to X$ in $\MET$, for every $u\in[0,\infty]$. 

\begin{theorem}
Let $X=(X,d)$ be a metric space. Then the following assertions are equivalent.
\begin{eqcond}
\item $X$ is cocomplete.
\item $X$ has all order-theoretic suprema and is tensored and cotensored.
\item $X$ has all (order theoretic) suprema, is tensored and, for every $u\in[0,\infty]$, the monotone map $(-)+u:X_p\to X_p$ has a right adjoint in $\ORD$.
\item $X$ has all (order theoretic) suprema, is tensored and, for every $u\in[0,\infty]$, the monotone map $(-)+u:X_p\to X_p$ preserves suprema.
\item $X$ has all (order theoretic) suprema, is tensored and, for every $x\in X$, the monotone map $d(-,x):X_p\to[0,\infty]$ has a right adjoint in $\ORD$.
\item $X$ is has all (order theoretic) suprema, is tensored and, for every $x\in X$, the monotone map $d(-,x):X_p\to[0,\infty]$ preserves suprema.
\end{eqcond}
Under these conditions, the supremum of a ``down-set'' $\psi:X^\op\to[0,\infty]$ can be calculated as 
\begin{equation}\label{Eq:colimitSupMais}
 \Sup\psi=\inf_{x\in X}(x+\psi(x)).
\end{equation}
A metric  map $f:X\to Y$ between cocomplete metric spaces preserves all colimits if and only if it preserves tensors and suprema.
\end{theorem}
\begin{proof}
By the preceding discussion, the implications (i)$\Rw$(ii) and (ii)$\Rw$(iii) are obvious, and so are (iii)$\RLw$(iv) and (v)$\RLw$(vi). To see (iii)$\Rw$(v), just note that a right adjoint $(-)\trunminus u:X_p\to X_p$ of $(-)+u$ produces a right adjoint $x\trunminus(-):[0,\infty]\to X_p$ of $d(-,x)$. Finally, (vi)$\Rw$(i) can be shown by verifying that \eqref{Eq:colimitSupMais} calculates indeed a supremum of $\psi$.
\end{proof}

Every metric space $X=(X,d)$ induces metric  maps 
\begin{align*}
 X\oplus[0,\infty]\xrightarrow{\,\formalball_X\,}[0,\infty]^{X^\op} && \text{and} &&
 X^I\xrightarrow{\,\fammod_{X,I}\,}[0,\infty]^{X^\op}\qquad\text{(where $I$ is any set)}.
\end{align*}
Here $\formalball_X:X\oplus[0,\infty]\to[0,\infty]^{X^\op},(x,u)\mapsto d(-,x)+u$ is the mate of the composite
\[
X^\op\oplus X\oplus [0,\infty]\xrightarrow{\,d\oplus 1\,}
[0,\infty]\oplus [0,\infty]\xrightarrow{\,+\,}[0,\infty],
\]
and $\fammod_{X,I}:X^I\to[0,\infty]^{X^\op}$ is the mate of the composite
\[
 X^\op\oplus X^I\to[0,\infty]^I\xrightarrow{\,\inf\,}[0,\infty],
\]
where the first component is the mate of the composite
\[
 X^\op\oplus X^I\oplus I\xrightarrow{\,1\oplus\ev\,} X^\op\oplus X
\xrightarrow{\,d\,}[0,\infty].
\]
Spelled out, for $\varphi\in X^I$ and $x\in X$, $\fammod_{X,I}(\varphi)(x)=\inf_{i\in I}d(x,\varphi(i))$, and a supremum of $\fammod_{X,I}(\varphi)\in [0,\infty]^{X^\op}$ is also a (order-theoretic) supremum of the family $(\varphi(i))_{i\in I}$ in $X$.
If $X$ is cocomplete, by composing with $\Sup_X$ one obtains metric  maps
\begin{align*}
X\oplus[0,\infty]\xrightarrow{\,+\,}X && && \text{and} &&
X^I\xrightarrow{\,\bigvee\,}X\qquad\text{(where $I$ is any set)}.
\end{align*}

Finally, a categorical standart argument (see \citep[Lemma 4.10]{Joh_StoneSp} shows that with $Y$ also $Y^X$ is injective, hence, $Y^X$ is cocomplete. Furthermore, tensors and suprema in $Y^X$ can be calculated pointwise:
\begin{align*}
h+u=(-+u)\cdot h &&\text{and}&& \left(\bigvee_{i\in I}h_i\right)=\bigvee\cdot\langle h_i\rangle_{i\in I},
\end{align*}
for $u\in[0,\infty]$, $h\in Y^X$ and $h_i\in Y^X$ ($i\in I$). Here $\langle h_i\rangle_{i\in I}:X\to Y^I$ denotes the map induced by the family $(h_i)_{i\in I}$.

\subsection{$[0,\infty]$-actions on ordered sets}

When $X=(X,d)$ is a tensored  metric space, we might not have $\Sup_X$ defined on the whole space $[0,\infty]^{X^\op}$, but it still is defined on its subspace of all metric maps $\psi:X^\op\to[0,\infty]$ of the form $\psi=d(-,x)+u$. Hence, one still has a metric  map
\[
 X\oplus[0,\infty]\to X,\,(x,u)\mapsto x+u,
\]
and one easily verifies the following properties.
\begin{enumerate}
\item\label{cond:action1} For all $x\in X$, $x+0\simeq x$.
\item\label{cond:action2} For all $x\in X$ and all $u,v\in[0,\infty]$, $(x+u)+v\simeq x+(u+v)$.
\item\label{cond:action3} $+:X_p\times[0,\infty]\to X_p$ is monotone in the first and anti-monotone in the second variable.
\item\label{cond:action4} For all $x\in X$, $x+\infty$ is a bottom element of $X_p$.
\item\label{cond:action5} For all $x\in X$ and $(u_i)_{i\in I}$ in $[0,\infty]$, $\displaystyle{x+\inf_{i\in I}u_i\simeq\bigvee_{i\in I}(x+u_i)}$.
\end{enumerate}
Of course, \eqref{cond:action4} is a special case of \eqref{cond:action5}. If $X$ is separated, then the first three conditions just tell us that $X_p$ is an algebra for the monad induced by the monoid $([0,\infty],\geqslant,+,0)$ on $\ORD_\sep$. Hence, $X\mapsto X_p$ defines a forgetful functor
\[
\MET_{\sep,+}\to\ORD^{[0,\infty]}_\sep,
\]
where $\MET_{\sep,+}$ denotes the category of tensored and separated metric spaces and tensor preserving metric  maps, and $\ORD^{[0,\infty]}_\sep$ the category of separated ordered sets with an unitary (i.e.\ satisfying \eqref{cond:action1}) and associative (i.e.\ satisfying \eqref{cond:action2}) action of $([0,\infty],\geqslant,+,0)$ ($[0,\infty]$-algebras for short) and monotone maps which preserve this action.

Conversely, let now $X$ be a $[0,\infty]$-algebra with action $+:X\times[0,\infty]\to X$. We define
\begin{equation}\label{Eq:metric_from_action}
d(x,y)=\inf\{u\in[0,\infty]\mid x+u\le y\}.
\end{equation}
Certainly, $x\le y$ implies $0\geqslant d(x,y)$, in particular one has $0\geqslant d(x,x)$. Since, for $x,y,z\in X$,
\begin{align*}
d(x,y)+d(y,z)
&=\inf\{u\in[0,\infty]\mid x+u\le y\}+\inf\{v\in[0,\infty]\mid y+v\le z\}\\
&=\inf\{u+v\mid u,v\in[0,\infty],\, x+u\le y,\, y+v\le z\}\\
&\geqslant\inf\{w\in[0,\infty]\mid x+w\le z=d(x,z),
\end{align*}
we have seen that $(X,d)$ is a metric space. If the $[0,\infty]$-algebra $X$ comes from a tensored separated metric space, then we get the original metric back. If $X$ satisfies \eqref{cond:action4}, then the infimum in \eqref{Eq:metric_from_action} is non-empty, and therefore
\begin{align*}
d(x+u,y)
&= \inf\{v\in[0,\infty]\mid x+u+v\le y\}\\
&= \inf\{w\trunminus u\mid w\in[0,\infty],\, x+w\le y\}\\
&= \inf\{w\mid w\in[0,\infty],\, x+w\le y\}\trunminus u
= d(x,y)\trunminus u,
\end{align*}
hence $(X,d)$ is tensored where $+$ is given by the algebra operation. Finally, if $X$ satisfies \eqref{cond:action5}, then the infimum in \eqref{Eq:metric_from_action} is actually a minimum, and therefore $0\geqslant d(x,y)$ implies $x\le y$. All told:
\begin{theorem}\label{Thm:TensMetVsAction}
The category $\MET_{\sep,+}$ is equivalent to the full subcategory of $\ORD^{[0,\infty]}_\sep$ defined by those $[0,\infty]$-algebras satisfying (5). Under this correspondence, $(X,d)$ is a cocomplete separated metric space if and only if the $[0,\infty]$-algebra $X$ has all suprema and $(-)+u:X\to X$ preserves suprema, for all $u\in[0,\infty]$.
\end{theorem}
\begin{remark}
The second part of the theorem above is essentially in \citep{PT89}, which actually states that cocomplete separated metric spaces correspond precisely to sup-lattices equipped with an unitary and associative action $+:X\times[0,\infty]\to X$ which is a bimorphism, meaning that it preserves suprema in each variable (where the order on $[0,\infty]$ is $\geqslant$) but not necessarily in both. Thanks to Freyd's Adjoint Functor Theorem (see \citep[Section V.6]{MacLane_WorkMath}), the category $\SUP$ of sup-lattices and suprema preserving maps admits a tensor product $X\otimes Y$ which is characterised by
\[
 \Bimorph(X\times Y,Z)\simeq\SUP(X\otimes Y,Z),
\]
naturally in $Z\in\SUP$, for all sup-lattices $X,Y$. Hence, a cocomplete separated metric space can be identified with a sup-lattice $X$ equipped with an unitary and associative action $+:X\otimes[0,\infty]\to X$ in $\SUP$. 
\end{remark}
\begin{proposition}\label{Prop:ContrTensMet}
Let $X$ and $Y$ be tensored metric spaces and $f:X\to Y$ be a map. Then $f:X\to Y$ is a metric  map if and only if $f:X_p\to Y_p$ is monotone and, for all $x\in X$ and $u\in[0,\infty]$, $f(x)+u\le f(x+u)$.
\end{proposition}
\begin{proof}
Every metric  map is also monotone with respect to the underlying orders and satisfies $f(x)+u\le f(x+u)$, for all $x\in X$ and $u\in[0,\infty]$. To see the reverse implication, recall that the metric $d$ on $X$ satisfies
\[
 d(x,y)=\inf\{u\in[0,\infty]\mid x+u\le y\},
\]
and for the metric $d'$ on $Y$ one has
\[
 d(f(x),f(y))=\inf\{v\in[0,\infty]\mid f(x)+v\le f(y)\}.
\]
If $x+u\le y$, then $f(x)+u\le f(x+u)\le f(y)$, and the assertion follows.
\end{proof}

\section{Metric compact Hausdorff spaces and approach spaces}

\subsection{Continuous lattices}\label{Subsect:ContLat}

Continuous lattices were introduced by D.\ Scott \citep{Sco_ContLat} as precisely those orders appearing as the \emph{specialisation order} of an injective topological T$_0$-space. Here, for an arbitrary topological space $X$ with topology $\calO$, the specialisation order $\le$ on $X$ is defined as
\[
 x\le y\quad\text{whenever}\quad\calO(x)\subseteq\calO(y),
\]
for all $x,y\in X$. This relation is always reflexive and transitive, and it is anti-symmetric if and only if $X$ is T$_0$. If $X$ is an injective T$_0$-space, then the ordered set $(X,\le)$ is actually complete and, for all $x\in X$,
\begin{equation*}
 x=\bigvee\{y\in X\mid y\ll x\}
\end{equation*}
(where $y\ll x$ whenever $x\le\bigvee D\,\Rw\,y\in D$ for every up-directed down-set $D\subseteq X$); in general, a complete separated ordered set with this property is called \emph{continuous lattice}. In this particular case the specialisation order contains all information about the topology of $X$: $A\subseteq X$ is open if and only if $A$ is unreachable by up-directed suprema in the sense that 
\begin{equation}\label{eq:ScottTopology}
 \bigvee D\in A\,\Rw\, D\cap A\neq\varnothing
\end{equation}
for every up-directed down-set $D\subseteq A$. Quite generally, \eqref{eq:ScottTopology} defines a topology on $X$ for any ordered set $X$, and a monotone map $f:X\to Y$ is continuous with respect to these topologies if and only if $f$ preserves all existing up-directed suprema. Furthermore, the specialisation order of this topology gives the original order back, and one obtains an injective topological T$_0$-space if and only if $X$ is a continuous lattice.

In the sequel we will consider topological spaces mostly via ultrafilter convergence, and therefore define the \emph{underlying order} $\le$ of a topological space as the ``point shadow'' of this convergence:
\begin{align*}
 x\le y &\quad\text{whenever}\quad\doo{x}\to y, &&(\doo{x}=\{A\subseteq X\mid x\in A\})
\end{align*}
which is dual to the specialisation order. Consequently, the underlying order of an injective topological T$_0$ space is an \emph{op-continuous lattice} meaning that $(X,\le)$ is complete and, for any $x\in X$,
\[
 x=\bigwedge\{y\in X\mid y\succ x\},
\]
where $y\succ x$ whenever $x\le\bigwedge D$ implies $y\in D$ for every down-directed up-set $D\subseteq X$. We also note that, with respect to the underlying order, the convergence relation of an injective T$_0$ space is given by
\begin{align*}
 \frx\to x &\iff \left(\bigwedge_{A\in\frx}\bigvee_{y\in A}y\right)\le x,
\end{align*}
for all ultrafilters $\frx$ on $X$ and all $x\in X$.

\subsection{Ordered compact Hausdorff spaces}\label{Subsect:OrdCompHaus}

The class of domains with arguably the most direct generalisation to metric spaces is that of stably compact spaces, or equivalently, \emph{ordered compact Hausdorff spaces}. The latter were introduced by \citep{Nach_TopOrd} as triples $(X,\le,\calO)$ where $(X,\le)$ is an ordered set (we do not assume anti-symmetry here) and $\calO$ is a compact Hausdorff topology on $X$ so that $\{(x,y)\mid x\le y\}$ is closed in $X\times X$. A morphism of ordered compact Hausdorff spaces is a map $f:X\to Y$ which is both monotone and continuous. The resulting category of ordered compact Hausdorff space and morphisms we denote as $\ORDCH$. If $(X,\le,\calO)$ is an ordered compact Hausdorff spaces, then the dual order $\le^\circ$ on $X$ together with the topology $\calO$ defines an ordered compact Hausdorff spaces $(X,\le^\circ,\calO)$, and one obtains a functor $(-)^\op:\ORDCH\to\ORDCH$ which commutes with the canonical forgetful functor $\ORDCH\to\SET$.

Analogously to the fact that compact Hausdorff spaces and continuous maps form an algebraic category over $\SET$ via ultrafilter convergence $UX\to X$ \citep{Man_TriplCompAlg}, it is shown in \citep{Fla97a} that the full subcategory $\ORDCH_\sep$ of $\ORDCH$ defined by those spaces with anti-symmetric order is the category of Eilenberg-Moore algebras for the prime filter monad of up-sets on $\ORD_\sep$. The situation does not change much when we drop anti-symmetry, in \citep{Tho_OrderedTopStr} it is shown that ordered compact Hausdorff spaces are precisely the Eilenberg-Moore algebras for the ultrafilter monad $\mU=\umonad$ suitably defined on $\ORD$. Here the functor $U:\ORD\to\ORD$ sends an ordered set $X=(X,\le)$ to the set $UX$ of all ultrafilters on the set $X$ equipped with the order relation
\[
\frx\le\fry\quad\text{ whenever }\quad \forall A\in\frx,B\in\fry\,\exists x\in A,y\in B\,.\,x\le y;
\qquad (\frx,\fry\in UX)
\]
and the maps
\begin{align*}
 e_X:X &\to UX & m_X:UUX &\to UX\\
x &\mapsto \doo{x}:=\{A\subseteq X\mid x\in A\} &
\frX &\mapsto\{A\subseteq X\mid A^\#\in\frX\} 
\end{align*}
(where $A^\#:=\{\frx\in UX\mid A\in\frx\}$) are monotone with respect to this order relation. Then, for $\alpha:UX\to X$ denoting the convergence of the compact Hausdorff topology $\calO$, $(X,\le,\calO)$ is an ordered compact Hausdorff space if and only if $\alpha:U(X,\le)\to(X,\le)$ is monotone.

\subsection{Metric compact Hausdorff spaces}

The presentation in \citep{Tho_OrderedTopStr} is even more general and gives also an extension of the ultrafilter monad $\mU$ to $\MET$. For a metric space $X=(X,d)$ and ultrafilters $\frx,\fry\in UX$, one defines a distance
\[
Ud(\frx,\fry)=\sup_{A\in\frx,B\in\fry}\inf_{x\in A,y\in B}d(x,y)
\]
and turns this way $UX$ into a metric space. Then $e_X:X\to UX$ and $m_X:UUX\to UX$ are metric maps and $Uf:UX\to UY$ is a metric map if $f:X\to Y$ is so. Not surprisingly, we call an Eilenberg--Moore algebra for this monad \emph{metric compact Hausdorff space}. Such a space can be described as a triple $(X,d,\alpha)$ where $(X,d)$ is a metric space and $\alpha$ is (the convergence relation of) a compact Hausdorff topology on $X$ so that $\alpha:U(X,d)\to(X,d)$ is a metric map. We denote the category of metric compact Hausdorff spaces and morphisms (i.e.\ maps which are both metric maps and continuous) as $\METCH$. The operation ``taking the dual metric space'' lifts to an endofunctor $(-)^\op:\METCH\to\METCH$ where $X^\op:=(X,d^\circ,\alpha)$, for every metric compact Hausdorff space $X=(X,d,\alpha)$.

\begin{example}\label{Ex:PasMCHsp}
The metric space $[0,\infty]$ with metric $\mu(u,v)=v\trunminus u$ becomes a metric compact Hausdorff space with the Euclidean compact Hausdorff topology whose convergence is given by
\[
\xi(\frv)=\sup_{A\in\frv}\inf_{v\in A}v,
\]
for $\frv\in U[0,\infty]$. Consequently, $[0,\infty]^\op$ denotes the metric compact Hausdorff space $([0,\infty],\mu^\circ,\xi)$ with the same compact Hausdorff topology on $[0,\infty]$ and with the metric $\mu^\circ(u,v)=u\trunminus v$.
\end{example}
\begin{lemma}\label{Lem:UXisTensored}
If $(X,d)$ is a tensored metric space, then $(UX,Ud)$ is tensored too.
\end{lemma}
\begin{proof}
For $u\in[0,\infty]$ and $\frx\in UX$, put $\frx+u=U(t_u)(\frx)$ where $t_u:X\to X$ sends $x\in X$ to (a choice of) $x+u$. Then
\begin{align*}
 Ud(\frx+u,\fry)&=\sup_{A\in\frx,B\in\fry}\inf_{x\in A,y\in B}d(x+u,y)\\
&=\left(\sup_{A\in\frx,B\in\fry}\inf_{x\in A,y\in B}d(x,y)\right)\trunminus u\\
&=Ud(\frx,\fry)\trunminus u,
\end{align*}
for all $\fry\in UX$. Here we use the fact that $-\trunminus u:[0,\infty]\to[0,\infty]$ preserves all suprema and non-empty infima.
\end{proof}
Clearly, if $f:X\to Y$ is a tensor preserving map between tensored metric spaces, then $Uf(\frx+u)\simeq Uf(\frx)+u$, hence $U:\MET\to\MET$ restricts to an endofunctor on the category $\MET_+$ of tensored metric spaces and tensor preserving maps.

\subsection{Stably compact topological spaces}\label{Subsect:StablyCompSp}

As we have already indicated at the beginning of Subsection \ref{Subsect:OrdCompHaus}, (anti-symmetric) ordered compact Hausdorff spaces can be equivalently seen as special topological spaces. In fact, both structures of an ordered compact Hausdorff space $X=(X,\le,\calO)$ can be combined to form a topology on $X$ whose opens are precisely those elements of $\calO$ which are down-sets in $(X,\le)$, and this procedure defines indeed a functor $K:\ORDCH\to\TOP$. An ultrafilter $\frx\in UX$ converges to a point $x\in X$ with respect to this new topology if and only if $\alpha(\frx)\le x$, where $\alpha:UX\to X$ denotes the convergence of $(X,\calO)$. Hence, $\le$ is just the underlying order of $\calO$ and $\alpha(\frx)$ is a smallest convergence point of $\frx\in UX$ with respect to this order. From that it follows at once that we can recover both $\le$ and $\alpha$ from $\calO$. To be rigorous, this is true when $(X,\le)$ is anti-symmetric, in the general case $\alpha$ is determined only up to equivalence. In any case, we define the dual of a topological space $Y$ of the form $Y=K(X,\le,\alpha)$ as $Y^\op=K(X,\le^\circ,\alpha)$, and note that equivalent maps $\alpha$ lead to the same space $Y^\op$.

A T$_0$ space $X=(X,\calO)$ comes from a anti-symmetric ordered compact Hausdorff space precisely if $X$ is \emph{stably compact}, that is, $X$ is sober, locally compact and stable. The latter property can be defined in different manners, we use here the one given in \citep{Sim82a}: $X$ is stable if, for open subsets $U_1,\ldots,U_n$ and $V_1,\ldots,V_n$ ($n\in\N$) of $X$ with $U_i\ll V_i$ for each $1\le i\le n$, also $\bigcap_i U_i\ll\bigcap_i V_i$. As usual, it is enough to require stability under empty and binary intersections, and stability under empty intersection translates to compactness of $X$. Also note that a T$_0$ space is locally compact if and only it is exponentiable in $\TOP$. It is also shown in \citep[Lemma 3.7]{Sim82a} that, for $X$ exponentiable, $X$ is stable if and only if, for every ultrafilter $\frx\in UX$, the set of all limit points of $\frx$ is irreducible\footnote{Actually, Lemma 3.7 of \citep{Sim82a} states only one implication, but the other is obvious and even true without assuming exponentiability.}. For a nice introduction to these kinds of spaces we refer to \citep{Jung04}. 

If we start with a metric compact Hausdorff space $X=(X,d,\alpha)$ instead, the construction above produces, for every $\frx\in UX$ and $x\in X$, the \emph{value of convergence}
\begin{equation}\label{Eq:ValConvMetCompHaus}
 a(\frx,x)=d(\alpha(\frx),x)\in[0,\infty],
\end{equation}
which brings us into the realm of

\subsection{Approach spaces}\label{Subsect:App}

We will here give a quick overview of \emph{approach spaces} which were introduced in \citep{Low_Ap} and are extensively described in \citep{Low_ApBook}. An approach space is typically defined as a pair $(X,\delta)$ consisting of a set $X$ and an \emph{approach distance} $\delta$ on $X$, that is, a function $\delta:X\times \two^X\to[0,\infty]$ satisfying
\begin{enumerate}
\item $\delta(x,\{x\})=0$,
\item $\delta(x,\varnothing)=\infty$,
\item $\delta(x,A\cup B)=\min\{\delta(x,A),\delta(x,B)\}$,
\item $\delta(x,A)\leqslant\delta(x,A^{(\eps)})+\eps$, where $A^{(\eps)}=\{x\in X\mid \delta(x,A)\leqslant\eps\}$,
\end{enumerate} 
for all $A,B\subseteq X$, $x\in X$ and $\eps\in[0,\infty]$. For $\delta:X\times \two^X\to[0,\infty]$ and $\delta':Y\times\two^Y\to[0,\infty]$, a map $f:X\to Y$ is called \emph{approach map} $f:(X,\delta)\to(Y,\delta')$ if $\delta(x,A)\geqslant\delta'(f(x),f(A))$, for every $A\subseteq X$ and $x\in X$. Approach spaces and approach maps are the objects and morphisms of the category $\AP$. The canonical forgetful functor
\[
 \AP\to\SET
\]
is topological, hence $\AP$ is complete and cocomplete and $\AP\to\SET$ preserves both limits and colimits. Furthermore, the functor $\AP\to\SET$ factors through $\TOP\to\SET$ where $(-)_p:\AP\to\TOP$ sends an approach space $(X,\delta)$ to the topological space with the same underlying set $X$ and with 
\[
x\in\overline{A} \text{ whenever } \delta(x,A)=0.
\]
This functor has a left adjoint $\TOP\to\AP$ which one obtains by interpreting the closure operator of a topological space $X$ as
\[
 \delta(x,A)=
\begin{cases}
 0 & \text{if $x\in\overline{A}$,}\\
\infty & \text{else.}
\end{cases}
\]
In fact, the image of this functor can be described as precisely those approach spaces where $\delta(x,A)\in\{0,\infty\}$, for all $x\in X$ and $A\subseteq X$. Being left adjoint, $\TOP\to\AP$ preserves all colimits, and it is not hard to see that this functor preserves also all limits (and hence has a left adjoint).

As in the case of topological spaces, approach spaces can be described in terms of many other concepts such as ``closed sets'' or convergence. For instance, every approach distance $\delta:X\times\two^X\to[0,\infty]$ defines a map
\[
a:UX\times X\to[0,\infty],\,a(\frx,x)=\sup_{A\in\frx}\delta(x,A),
\]
and vice versa, every $a:UX\times X\to[0,\infty]$ defines a function
\[
\delta:X\times\two^X\to[0,\infty],\,\delta(x,A)=\inf_{A\in\frx}a(\frx,x).
\]
Furthermore, a mapping $f:X\to Y$ between approach spaces $X=(X,a)$ and $Y=(Y,b)$ is an approach map if and only if $a(\frx,x)\geqslant b(Uf(\frx),f(x))$, for all $\frx\in UX$ and $x\in X$. Therefore one might take as well convergence as primitive notion, and axioms characterising those functions $a:UX\times X\to[0,\infty]$ coming from a approach distance can be already found in \citep{Low_Ap}. In this paper we will make use the characterisation (given in \citep{CH_TopFeat}) as precisely the functions $a:UX\times X\to[0,\infty]$ satisfying 
\begin{align}\label{Eq:AxiomsApp}
0\geqslant a(\doo{x},x) &&\text{and}&& Ua(\frX,\frx)+a(\frx,x)\geqslant a(m_X(\frX),x),
\end{align}
where $\frX\in UUX$, $\frx\in UX$, $x\in X$ and
\[
 Ua(\frX,\frx)=\sup_{\calA\in\frX,A\in\frx}\inf_{\fra\in\calA,x\in A}a(\fra,x).
\]
In the language of convergence, the underlying topological space $X_p$ of an approach space $X=(X,a)$ is defined by  $\frx\to x\iff a(\frx,x)=0$, and a topological space $X$ can be interpreted as an approach space by putting $a(\frx,x)=0$ whenever $\frx\to x$ and $a(\frx,x)=\infty$ otherwise.

We can restrict $a:UX\times X\to[0,\infty]$ to principal ultrafilters and obtain a metric
\[
 a_0:X\times X\to[0,\infty],\,(x,y)\mapsto a(\doo{x},y)
\]
on $X$. Certainly, an approach map is also a metric map, therefore this construction defines a functor
\[
 (-)_0:\AP\to\MET.
\]
which, combined with $(-)_p:\MET\to\ORD$, yields a functor
\[
 \AP\to\ORD
\]
where $x\le y$ whenever $0\geqslant a(\doo{x},y)$. This order relation extends point-wise to approach maps, and we can consider $\AP$ as an ordered category. As before, this additional structure allows us to speak about adjunction in $\AP$: for approach maps $f:(X,a)\to (X',a')$ and $g:(X',a')\to(X,a)$, $f\dashv g$ if $1_X\le g\cdot f$ and $f\cdot g\le 1_{X'}$; equivalently, $f\dashv g$ if and only if, for all $\frx\in UX$ and $x'\in X'$,
\[
 a'(Uf(\frx),x')=a(\frx,g(x')).
\]
One calls an approach space $X=(X,a)$ \emph{separated}, or T$_0$, if the underlying topology of $X$ is T$_0$, or, equivalently, if the underlying metric of $X$ is separated. Note that this is the case precisely if, for all $x,y\in X$, $a(\doo{x},y)=0=a(\doo{y},x)$ implies $x=y$. 

Similarly to the situation for metric spaces, besides the categorical product there is a further approach structure on the set
\[
 X\times Y
\]
for approach spaces $X=(X,a)$ and $Y=(Y,b)$, namely
\[
 c(\frw,(x,y))=a(\frx,x)+b(\fry,y)
\]
where $\frw\in U(X\times Y)$, $(x,y)\in X\times Y$ and $\frx=U\pi_1(\frw)$ and $\fry=U\pi_2(\frw)$. The resulting approach space $(X\times Y,c)$ we denote as $X\oplus Y$, in fact, one obtains a functor $\oplus:\AP\times\AP\to\AP$. We also note that $1\oplus X\simeq X\simeq X\oplus 1$, for every approach space $X$.

Unfortunately, the above described monoidal structure on $\AP$ is not closed, the functor $X\oplus-:\AP\to\AP$ does not have in general a right adjoint (see \citep{Hof_TopTh}). If it does, we say that the approach space $X=(X,a)$ is \emph{$+$-exponentiable} and denote this right adjoint as $(-)^X:\AP\to\AP$. Then, for any approach space $Y=(Y,b)$, the space $Y^X$ can be chosen as the set of all approach maps of type $X\to Y$, equipped with the convergence
\begin{equation}\label{Eq:FunSpStr}
\fspstr{\frp}{h}=\sup\{b(U\!\ev(\frw),h(x))\trunminus a(\frx,x)\mid x\in X,\frw\in U(Y^X\oplus X)\text{ with }\frw\mapsto\frp, (\frw\mapsto\frx)\},
\end{equation}
for all $\frp\in U(Y^X)$ and $h\in Y^X$. If $\frp=\doo{k}$ for some $k\in Y^X$, then
\[
 \fspstr{\doo{k}}{h}=\sup_{x\in X}b_0(k(x),h(x)),
\]
which tells us that $(Y^X)_0$ is a subspace of $Y^{X_0}$. If $X=(X,a)$ happens to be topological, i.e.\ $a$ only takes values in $\{0,\infty\}$, then \eqref{Eq:FunSpStr} simplifies to
\[
 \fspstr{\frp}{h}=\sup\{b(U\!\ev(\frw),h(x))\mid x\in X,\frw\in U(Y^X\oplus X)\text{ with }\frw\mapsto\frp, (\frw\mapsto\frx),a(\frx,x)=0\}.
\]
Furthermore, a topological approach space is $+$-exponentiable if and only if it is exponentiable in $\TOP$, that is, core-compact. This follows for instance from the characterisation of exponentiable topological spaces given in \citep{Pis_ExpTop}, together with the characterisation of $+$-exponentiable approach spaces \citep{Hof_TopTh} as precisely the ones where the convergence structure $a:UX\times X\to[0,\infty]$ satisfies
\[
 a(m_X(\frX),x)=\inf_{\frx\in X}(Ua(\frX,\frx)+a(\frx,x)),
\]
for all $\frX\in UUX$ and $x\in X$. Note that the left hand side is always smaller or equal to the right hand side.

Via the embedding $\TOP\to\AP$ described earlier in this subsection, which is left adjoint to $(-)_p:\AP\to\TOP$, we can interpret every topological space $X$ as an approach space, also denoted as $X$, where the convergence structure takes only values in $\{0,\infty\}$. Then, for any approach space $Y$, $X\oplus Y=X\times Y$, which in particular tells us that the diagram
\[
 \xymatrix{\AP\ar[r]^{X\oplus-} & \AP\\
\TOP\ar[u]\ar[r]_{X\times-} &\TOP\ar[u]}
\]
commutes. Therefore, if $X$ is core-compact, then also the diagram of the corresponding right adjoints commutes, hence
\begin{lemma}\label{Lem:PowerVsUnderlyingTop}
For every core-compact topological space $X$ and every approach space $Y$, $(Y^X)_p=(Y_p)^X$.
\end{lemma}
\begin{remark}
To be rigorous, the argument presented above only allows us to conclude $(Y^X)_p\simeq(Y_p)^X$. However, since we can choose the right adjoints $(-)^X$ and $(-)_p$ exactly as described earlier, one has indeed equality.
\end{remark}

The lack of good function spaces can be overcome by moving into a larger category where these constructions can be carried out. In the particular case of approach spaces, a good environment for doing so is the category $\PSAP$ of \emph{pseudo-approach spaces} and approach maps \citep{LL_TopHulls}. Here a pseudo-approach space is pair $X=(X,a)$ consisting of a set $X$ and a convergence structure $a:UX\times X\to[0,\infty]$ which only needs to satisfy the first inequality of \eqref{Eq:AxiomsApp}: $0\geqslant a(\doo{x},x)$, for all $x\in X$. If $X=(X,a)$ and $Y=(Y,b)$ are pseudo-approach spaces, then one defines $X\oplus Y$ exactly as for approach spaces, and the formula \eqref{Eq:FunSpStr} defines a pseudo-approach structure on the set $Y^X$ of all approach maps from $X$ to $Y$, without any further assumptions on $X$ or $Y$. In fact, this construction leads now to an adjunction $X\oplus-\dashv(-)^X:\PSAP\to\PSAP$, for every pseudo-approach space $X=(X,a)$.

\subsection{Stably compact approach spaces}

Returning to metric compact Hausdorff spaces, one easily verifies that \eqref{Eq:ValConvMetCompHaus} defines an approach structure on $X$ (see \citep{Tho_OrderedTopStr}, for instance). Since a homomorphism between metric compact Hausdorff spaces becomes an approach map with respect to the corresponding approach structures, one obtains a functor
\[
K:\METCH\to\AP.
\]
The underlying metric of $KX$ is just the metric $d$ of the metric compact Hausdorff space $X=(X,d,\alpha)$, and $x=\alpha(\frx)$ is a \emph{generic convergence point} of $\frx$ in $KX$ in the sense that
\[
 a(\frx,y)=d(x,y),
\]
for all $y\in X$. The point $x$ is unique up to equivalence since, if one has $x'\in X$ with the same property, then
\[
 d(x,x')=a(\frx,x')=d(x',x')=0
\]
and, similarly, $d(x',x)=0$. In analogy to the topological case, we introduce the \emph{dual} $Y^\op$ of an approach space $Y=K(X,d,\alpha)$ as $Y^\op=K(X,d^\circ,\alpha)$, and we call an T$_0$ approach space \emph{stably compact} if it is of the form $KX$, for some metric compact Hausdorff space $X$.
\begin{lemma}\label{Lem:MapsMetCompSp}
Let $(X,d,\alpha)$, $(Y,d',\beta)$ be metric compact Hausdorff spaces with corresponding approach spaces $(X,a)$ and $(Y,b)$, and let $f:X\to Y$ be a map. Then $f$ is an approach map $f:(X,a)\to(Y,b)$ if and only if $f:(X,d)\to(Y,d')$ is a metric map and $\beta\cdot Uf(\frx)\le f\cdot\alpha(\frx)$, for all $\frx\in UX$.
\end{lemma}
\begin{proof}
Assume first that $f:(X,d)\to(Y,d')$ is in $\MET$ and that $\beta\cdot Uf(\frx)\le f\cdot\alpha(\frx)$, for all $\frx\in UX$. Then
\[
 a(\frx,x)=d(\alpha(\frx),x)\geqslant d'(f\cdot\alpha(\frx),f(x))
\geqslant d'(\beta\cdot Uf(\frx),f(x))=b(Uf(\frx),f(x)).
\]
Suppose now that $f:(X,a)\to(Y,b)$ is in $\AP$ and let $\frx\in UX$. Then
\[
 0\geqslant d(\alpha(\frx),\alpha(\frx))=a(\frx,\alpha(\frx))\geqslant b(Uf(\frx),f\cdot \alpha(\frx))= d'(\beta\cdot Uf(\frx),f\cdot\alpha(\frx)).
\]
Clearly, $f:(X,a)\to(Y,b)$ in $\AP$ implies $f:(X,d)\to(Y,d')$ in $\MET$, and the assertion follows.
\end{proof}
It is an important fact that $K$ has a left adjoint
\[
M:\AP\to\METCH
\]
which can be described as follows (see \citep{Hof_DualityDistSp}). For an approach space $X=(X,a)$, $MX$ is the metric compact Hausdorff space with underlying set $UX$ equipped with the compact Hausdorff convergence $m_X:UUX\to UX$ and the metric
\begin{equation}\label{Eq:MetricOnUX}
 d:UX\times UX\to[0,\infty],\
(\frx,\fry)\mapsto\inf\{\eps\in[0,\infty]\mid \forall A\in\frx\,.\,A^{(\eps)}\in\fry\},
\end{equation}
and $Mf:=Uf:UX\to UY$ is a homomorphism between metric compact Hausdorff spaces provided that $f:X\to Y$ is an approach map between approach spaces. The unit and the counit of this adjunction are given by
\begin{align*}
e_X:(X,a)\to(UX,d(m_X(-),-)) &&\text{and}&&
\alpha:(UX,d,m_X)\to(X,d,\alpha)
\end{align*}
respectively, for $(X,a)$ in $\AP$ and $(X,d,\alpha)$ in $\METCH$.
\begin{remark}
All what was said here about metric compact Hausdorff spaces and approach space can be repeated, \emph{mutatis mutandis}, for ordered compact Hausdorff spaces and topological spaces. For instance, the funcor $K:\ORDCH\to\TOP$ (see Subsection \ref{Subsect:StablyCompSp}) has a left adjoint $M:\TOP\to\ORDCH$ which sends a topological space $X$ to $(UX,\le,m_X)$, where
\[
 \frx\le\fry\quad\text{whenever}\quad\forall A\in\frx\,.\,\overline{A}\in\fry,
\]
for all $\frx,\fry\in UX$. Furthermore, Lemma \ref{Lem:MapsMetCompSp} reads now as follows: Let $(X,\le,\alpha)$, $(Y,\le,\beta)$ be ordered compact Hausdorff spaces with corresponding topological spaces $(X,a)$ and $(Y,b)$, and let $f:X\to Y$ be a map. Then $f$ is a continuous map $f:(X,a)\to(Y,b)$ in $\TOP$ if and only if $f:(X,d)\to(Y,d')$ is in $\ORD$ and $\beta\cdot Uf(\frx)\le f\cdot\alpha(\frx)$, for all $\frx\in UX$.
\end{remark}

\begin{example}
The ordered set $\two=\{0,1\}$ with the discrete (compact Hausdorff) topology becomes an ordered compact Hausdorff space which induces the Sierpi\'nski space $\two$ where $\{1\}$ is closed and $\{0\}$ is open. Then the maps
\begin{enumerate}
\item $\bigwedge:\two^I\to\two$
\item $v\Rw-:\two\to\two$,
\item $v\wedge-:\two\to\two$
\setcounter{counter}{\value{enumi}}
\end{enumerate}
are continuous, for every set $I$ and $v\in\two$. Furthermore (see \citep{Nach_CompUnions,Esc_Synth}),
\begin{enumerate}
\setcounter{enumi}{\value{counter}}
\item $\bigvee:\two^I\to\two$ is continuous if and only if $I$ is a compact topological space.
\end{enumerate}
Here the function space $\two^I$ is possibly calculated in the category $\PSTOP$ of pseudotopological spaces (see \citep{HLCS_ImpTop}). In particular, if $I$ is a compact Hausdorff space, then $I$ is exponentiable in $\TOP$ and $\bigvee:\two^I\to\two$ belongs to $\TOP$.
\end{example}
\begin{example}
The metric space $[0,\infty]$ with distance $\mu(x,y)=y\trunminus x$ equipped with the Euclidean compact Hausdorff topology where $\frx$ converges to $\xi(\frx):=\sup_{A\in\frx}\inf A$ is a metric compact Hausdorff space (see Example \ref{Ex:PasMCHsp}) which gives the ``Sierpi\'nski approach space'' $[0,\infty]$ with approach convergence structure $\lambda(\frx,x)=x\trunminus\xi(\frx)$. Then, with the help of subsection \ref{Subsect:Prelim}, one sees that
\begin{enumerate}
\item $\sup:[0,\infty]^I\to[0,\infty]$,
\item $-\trunminus v:[0,\infty]\to[0,\infty]$,
\item $-+v:[0,\infty]\to[0,\infty]$
\setcounter{counter}{\value{enumi}}
\end{enumerate}
are approach maps, for every set $I$ and $v\in[0,\infty]$. If $I$ carries the structure $a:UI\times I\to[0,\infty]$ of an approach space, one defines the \emph{degree of compactness} \citep{Low_ApBook} of $I$ as
\[
 \comp(I)=\sup_{\frx\in UI}\inf_{x\in X}a(\frx,x).
\]
Then (see \citep{Hof_TopTh}), 
\begin{enumerate}
\setcounter{enumi}{\value{counter}}
\item $\inf:[0,\infty]^I\to[0,\infty]$ is an approach map if and only if $\comp(I)=0$.
\end{enumerate}
As above, the function space $[0,\infty]^I$ is possibly calculated in $\PSAP$, in fact, $(-)^I:\PSAP\to\PSAP$ is the right adjoint of $I\oplus-:\PSAP\to\PSAP$.
\end{example}

As any adjunction, $M\dashv K$ induces a monad on $\AP$ (respectively on $\TOP$). Here, for any approach space $X$, the space $KM(X)$ is the set $UX$ of all ultrafilters on the set $X$ equipped with an apporach structure, and the unit and the multiplication are essentially the ones of the ultrafilter monad. Therefore we denote this monad also as $\mU=\umonad$. In particular, one obtains a functor $U:=KM:\AP\to\AP$ (respectively $U:=KM:\TOP\to\TOP$). Surprisingly or not, the categories of algebras are equivalent to the Eilenberg--Moore categories on $\ORD$ and $\MET$:
\begin{align*}
\ORD^\mU\simeq\TOP^\mU &&\text{and}&& \MET^\mU\simeq\AP^\mU.
\end{align*}
More in detail (see \citep{Hof_DualityDistSp}), for any metric compact Hausdorff space $(X,d,\alpha)$ with corresponding approach space $(X,a)$, $\alpha:U(X,a)\to(X,a)$ is an approach contraction; and for an approach space $(X,a)$ with  Eilenberg--Moore algebra structure $\alpha:U(X,a)\to(X,a)$, $(X,d,\alpha)$ is a metric compact Hausdorff space where $d$ is the underlying metric of $a$ and, moreover, $a$ is the approach structure induced by $d$ and $\alpha$. 

It is worthwhile to note that the monad $\mU$ on $\TOP$ as well as on $\AP$ satisfies a pleasant technical property: it is of Kock-Z\"oberlein type \citep{Koc_MonAd,Zob_Doct}. In what follows we will not explore this further and refer instead for the definition and other information to \citep{EF_SemDom}. We just remark here that one important consequence of this property is that an Eilenberg--Moore algebra structure $\alpha:UX\to X$ on an $\{$approach, topological$\}$ space $X$ is necessarily left adjoint to $e_X:X\to UX$. If $X$ is T$_0$, then one even has that an approach map $\alpha:UX\to X$ is an Eilenberg--Moore algebra structure on $X$ if and only if $\alpha\cdot e_X=1_X$. Hence, a T$_0$ approach space $X=(X,a)$ is an $\mU$-algebra if and only if
\begin{enumerate}
\item\label{C3} every ultrafilter $\frx\in UX$ has a generic convergence point $\alpha(\frx)$ meaning that $a(\frx,x)=a_0(\alpha(\frx),x)$, for all $x\in X$, and
\item\label{C4} the map $\alpha:UX\to X$ is an approach map.
\end{enumerate}
We observed already in \citep{Hof_DualityDistSp} that the latter condition can be substituted by 
\begin{enumerate}
\renewcommand{\theenumi}{\ref{C4}'}
\item $X$ is $+$-exponentiable.
\end{enumerate}
For the reader familiar with the notion of sober approach space \citep{BLO_AFrm} we remark that the former condition can be splitted into the following two conditions:
\begin{enumerate}
\renewcommand{\theenumi}{\ref{C3}a}
\item for every ultrafilter $\frx\in UX$, $a(\frx,-)$ is an approach prime element, and
\renewcommand{\theenumi}{\ref{C3}b}
\item $X$ is sober. 
\end{enumerate}
Certainly, the two conditions above imply \eqref{C3}. For the reverse implication, just note that every approach prime element $\varphi:X\to[0,\infty]$ is the limit function of some ultrafilter $\frx\in UX$ (see \cite[Proposition 5.7]{BLO_AFrm}). Hence, every stably compact approach space is sober. We call an $+$-exponentiable approach space $X$ \emph{stable} if $X$ satisfies the condition (\ref{C3}a) above (compare with Subsection \ref{Subsect:StablyCompSp}), and with this nomenclature one has
\begin{proposition}
An T$_0$ approach space $X$ is stably compact if and only if $X$ is sober, $+$-exponentiable and stable.
\end{proposition}

\section{Injective approach spaces}

\subsection{Yoneda embeddings}\label{SubSect:YonedaInAPP}

Let $X=(X,a)$ be an approach space with convergence $a:UX\times X\to[0,\infty]$. Then $a$ is actually an approach map $a:(UX)^\op\oplus X\to[0,\infty]$, and we refer to its $+$-exponential mate $\yoneda_X:=\mate{a}:X\to[0,\infty]^{(UX)^\op}$ as the \emph{(covariant) Yoneda embedding} of $X$ (see \citep{CH_Compl} and \citep{Hof_DualityDistSp}). We denote the approach space $[0,\infty]^{(UX)^\op}$ as $PX$, and its approach convergence structure as $\fspstr{-}{-}$. One has $a(\frx,x)=\fspstr{U\!\yoneda_X(\frx)}{\yoneda_X(x)}$ for all $\frx\in UX$ and $x\in X$ (hence $\yoneda_X$ is indeed an embedding when $X$ is $T_0$) thanks to the \emph{Yoneda Lemma} which states here that, for all $\frx\in UX$ and $\psi\in PX$,
\[
 \fspstr{U\!\yoneda_X(\frx)}{\psi}=\psi(\frx).
\]

The metric $d:UX\times UX\to[0,\infty]$ (see \eqref{Eq:MetricOnUX}) is actually an approach map $d:(UX)^\op\oplus UX\to[0,\infty]$, whose mate can be seen as a \emph{``second'' (covariant) Yoneda embedding} $\yonedaT_X:UX\to PX$, and the ``second'' Yoneda Lemma reads as (see \citep{Hof_DualityDistSp})
\[
 \fspstr{U\!\yonedaT_X(\frX)}{\psi}=\psi(m_X(\frX)),
\]
for all $\frX\in UUX$ and $\psi\in PX$.
\begin{remark}\label{rem:FilterOfOpens}
Similarly, the convergence relation $\to:(UX)^\op\times X\to\two$ of a topological space $X$ is continuous, and by taking its exponential transpose we obtain the Yoneda embedding $\yoneda_X:X\to\two^{(UX)^\op}$. A continuous map $\psi:X^\op\to\two$ can be identified with a closed subset $\calA\subseteq UX$. In \citep{HT_LCls} it is shown that $\calA$ corresponds to a filter on the lattice of opens of $X$, moreover, the space $\two^{(UX)^\op}$ is homeomorphic to the space $F_0X$ of all such filters, where the topology on $F_0X$ has 
\begin{align*}
\{\frf\in F_0 X\mid A\in\frf\}&& \text{($A\subseteq X$ open)}
\end{align*}
as basic open sets (see \citep{Esc_InjSp}). Under this identification, the Yoneda embedding $\yoneda_X:X\to\two^{(UX)^\op}$ corresponds to the map $X\to F_0X$ sending every $x\in X$ to its neighbourhood filter, and $\yonedaT_X:UX\to F_0 X$ restricts an ultrafilter $\frx\in UX$ to its open elements.
\end{remark}

Since an approach space $X$ is in general not $+$-exponentiable, the set $[0,\infty]^X$ of all approach maps of type $X\to[0,\infty]$ does not admit a canonical approach structure. However, it still becomes a metric space when equipped with the sup-metric, that is, the metric space $[0,\infty]^X$ is a subspace of the $+$-exponential $[0,\infty]^{X_0}$ in $\MET$ of underlying metric space $X_0$ of $X$. Recall from Subsection \ref{SubSect:CocomplMetSp} that the contravariant Yoneda embedding $\yonedaOP_{X_0}:X_0\to\left([0,\infty]^{X_0}\right)^\op$ of the metric space $X_0$ sends an element $x\in X_0$ to the metric map $X_0\to[0,\infty],\,x'\mapsto a_0(x,x')=a(\doo{x},x')$. But the map $\yonedaOP_{X_0}(x)$ can be also seen as an approach map of type $X\to[0,\infty]$, hence this construction defines also a metric map $\yonedaOP_X:X_0\to\left([0,\infty]^X\right)^\op$, for every approach space $X$.

The inclusion map $[0,\infty]^X\hrw[0,\infty]^{X_0}$ has a left adjoint $[0,\infty]^{X_0}\to[0,\infty]^X$ which sends a metric map $\varphi:X_0\to[0,\infty]$ to the approach map $X\to[0,\infty]$ which sends $x$ to $\displaystyle{\inf_{\frx\in UX}a(\frx,x)+\xi(U\varphi(\frx))}$ (where $\displaystyle{\xi(\fru)=\sup_{A\in\fru}\inf_{u\in A}u}$). In particular, if $\varphi=a(e_X(x),-)$, then $\xi(U\varphi(\frx))=Ua(e_{UX}\cdot e_X(x),\frx)$ and therefore $\displaystyle{\inf_{\frx\in UX}a(\frx,x)+\xi(U\varphi(\frx))}=a(e_X(x),-)$. Hence, both the left and the right adjoint commute with the contravariant Yoneda embeddings.
\[
 \xymatrix{\left([0,\infty]^{X_0}\right)^\op\ar@/^1.5ex/[rr] &&
 \left([0,\infty]^X\right)^\op\ar@/^1.5ex/[ll]\\
 & X_0\ar[ul]^{\yonedaOP_{X_0}}\ar[ur]_{\yonedaOP_X}}
\]

Finally, one also have the \emph{Isbell conjugation adjunction} in this context:
\[
\xymatrix{\left([0,\infty]^X\right)^\op\ar@/^1.5ex/[rr]^{(-)^-}\ar@{}[rr]|\top &&
\left([0,\infty]^{X^\op}\right)_0\ar@/^1.5ex/[ll]^{(-)^+}\\
 & X_0\ar[ul]^{\yonedaOP_X}\ar[ur]_{\yoneda_X}}
\]
where
\begin{align*}
\varphi^-(\frx)&=\sup_{x\in X}(a(\frx,x)\trunminus\varphi(x))
&\text{and}&&
\psi^+(x)&=\sup_{\frx\in UX}(a(\frx,x)\trunminus\psi(\frx)).
\end{align*}
\begin{remark}
In our considerations above we were very sparse on details and proofs. This is because (in our opinion) this material is best presented in the language of \emph{modules} (also called distributors or profunctors), but we decided not to include this concept here and refer for details to \citep{CH_Compl} and \citep{Hof_Cocompl} (for the particular context of this paper) and to \citep{BenDistWork} and \citep{Law_MetLogClo}) for the general concept. We note that $\psi:(UX)^\op\to[0,\infty]$ is the same thing as a module $\psi:X\kmodto 1$ from $X$ to $1$ and $\varphi:X\to[0,\infty]$ is the same thing as a module $\varphi:1\kmodto X$. Then $\psi^+$ is the \emph{extension} of $\psi$ along the identity module on $X$ (see \citep[1.3 and Remark 1.5]{Hof_Cocompl}), and $\varphi^-$ is the \emph{lifting} of $\varphi$ along the identity module on $X$ (see \citep[Lemma 5.11]{HW_AppVCat}); and this process defines quite generally an adjunction.
\end{remark}

\subsection{Cocomplete approach spaces}

In this and the next subsection we study the notion of \emph{cocompleteness} for approach spaces, as initiated in \citep{CH_Compl,Hof_Cocompl,CH_CocomplII}. By analogy with ordered sets and metric spaces, we think of an approach map $\psi:(UX)^\op\to[0,\infty]$ as a ``down-set'' of $X$. A point $x_0\in X$ is a \emph{supremum} of $\psi$ if
\[
 a(\doo{x_0},x)=\sup_{\frx\in UX}a(\frx,x)\trunminus\psi(\frx),
\]
for all $x\in X$. As before, suprema are unique up to equivalence, and therefore we will often talk about the supremum. An approach map $f:(X,a)\to(Y,b)$ preserves the supremum of $\psi$ if
\[
 b(\doo{f(x_0)},y)=\sup_{\fry\in UY}b(Uf(\frx),y)\trunminus\psi(\frx).
\]
Not surprisingly (see \citep{Hof_Cocompl}),
\begin{lemma}
Left adjoint approach maps $f:X\to Y$ between approach spaces preserve all suprema which exist in $X$.
\end{lemma}

We call an approach space $X$ \emph{cocomplete} if every ``down-set'' $\psi:(UX)^\op\to[0,\infty]$ has a supremum in $X$. If this is the case, then ``taking suprema'' defines a map $\Sup_X:PX\to X$, indeed, one has
\begin{proposition}
An approach space $X$ is cocomplete if and only if $\yoneda_X:X_0\to(PX)_0$ has a left adjoint $\Sup_X:(PX)_0\to X_0$ in $\MET$.
\end{proposition}
\begin{remark}
We deviate here from the notation used in previous work where a space $X$ was called cocomplete whenever $\yoneda_X:X\to PX$ has a left adjoint in $\AP$. Approach spaces satisfying this (stronger) condition will be called absolutely cocomplete (see Subsection \ref{SubSect:OpContLatAct} below) here.
\end{remark}

With the help of Subsection \ref{SubSect:YonedaInAPP}, one sees immediately that $\Sup_X:(PX)_0\to X_0$ produces a left inverse of $\yonedaOP_{X_0}:X_0\to\left([0,\infty]^{X_0}\right)^\op$ in $\MET$, hence the underlying metric space $X_0$ is complete. Certainly, a left inverse of $ X_0\to\left([0,\infty]^{X_0}\right)^\op$ in $\MET$ gives a left inverse of $\yoneda_X:X_0\to(PX)_0$ in $\MET$, however, such a left inverse does not need to be a left adjoint (see Example \ref{ex:CocomplNotUCocompl}). In the following subsection we will see what is missing.

\subsection{Special types of colimits}

Similarly to what was done for metric spaces, we will be interested in approach spaces which admit certain types of suprema.

Our first example are \emph{tensored} approach spaces which are defined exactly as their metric counterparts. Explicitly, to every point $x$ of an approach space $X=(X,a)$ and every $u\in[0,\infty]$ one associates a ``down-set'' $\psi:(UX)^\op\to[0,\infty],\,\frx\mapsto a(\frx,x)+u$, and a supremum of $\psi$ (which, we recall, is unique up to equivalence) we denote as $x+u$. Then $X$ is called tensored if every such $\psi$ has a supremum in $X$. By definition, $x+u\in X$ is characterised by the equation
\[
 a(e_X(x+u),y)=\sup_{\frx\in UX}(a(\frx,y)\trunminus(a(\frx,x)+u)),
\]
for all $y\in X$. This supremum is actually obtained for $\frx=\doo{x}$, so that the right hand side above reduces to $a(\doo{x},y)\trunminus u$. Therefore:
\begin{proposition}
An approach space $X$ is tensored if and only if its underlying metric space $X_0$ is tensored.
\end{proposition}

We call an approach space $X=(X,a)$ \emph{$\mU$-cocomplete} if every ``down-set'' $\psi:(UX)^\op\to[0,\infty]$ of the form $\psi=\yonedaT_X(\frx)$ with $\frx\in UX$ has a supremum in $X$. Such a supremum, denoted as $\alpha(\frx)$, is characterised by
\[
 a(\doo{\alpha(\frx)},x)=\sup_{\fry\in UX}(a(\fry,x)\trunminus d(\fry,\frx)),
\]
for all $x\in X$. Here the supremum is obtained for $\fry=\frx$, hence the equality above translates to
\[
 a_0(\alpha(\frx),x)=a(\frx,x),
\]
where $a_0$ denotes the underlying metric of $a$. Since $a(\frx,x)=d(\frx,\doo{x})$ where $d$ is the metric on $(UX)_0$, we conclude that
\begin{proposition}
An approach space $X$ is $\mU$-cocomplete if and only if $e_X:X_0\to(UX)_0$ has a left adjoint $\alpha:(UX)_0\to X_0$ in $\MET$.
\end{proposition}
Note that every metric compact Hausdorff space is $\mU$-cocomplete. We are now in position to characterise cocomplete approach spaces.
\begin{theorem}
Let $X$ be an approach space. Then the following assertions are equivalent.
\begin{eqcond}
\item\label{C2} $X$ is cocomplete.
\item The metric space $X_0$ is complete and and the approach space $X$ is $\mU$-cocomplete.
\item\label{C1} The metric space $X_0$ is complete and $e_X:X_0\to(UX)_0$ has a left adjoint $\alpha:(UX)_0\to X_0$ in $\MET$.
\end{eqcond}
Furthermore, in this situation the supremum of a ``down-set'' $\psi:(UX)^\op\to[0,\infty]$ is given by
\begin{equation}\label{Eq:FormWeightSupAp}
\bigvee_{\frx\in UX}\alpha(\frx)+\psi(\frx).
\end{equation}
\end{theorem}
\begin{proof}
To see the implication \eqref{C1}$\Rw$\eqref{C2}, we only need to show that the formula \eqref{Eq:FormWeightSupAp} gives indeed a supremum of $\psi$. In fact, 
\[
a_0(\bigvee_{\frx\in UX}\alpha(\frx)+\psi(\frx),x)
=\sup_{\frx\in UX}a_0(\alpha(\frx)+\psi(\frx),x)
=\sup_{\frx\in UX}(a_0(\alpha(\frx),x)\trunminus\psi(\frx))
=\sup_{\frx\in UX}(a(\frx,x)\trunminus\psi(\frx)),
\]
for all $x\in X$.
\end{proof}
\begin{example}\label{ex:CocomplNotUCocompl}
Every metric compact Hausdorff space whose underlying metric is cocomplete gives rise to a cocomplete approach space. In particular, both $[0,\infty]$ and $[0,\infty]^\op$ are cocomplete (see Example \ref{Ex:PasMCHsp}).

To each metric $d$ on a set $X$ one associates the approach convergence structure 
\[
 a_d(\frx,y)=\sup_{A\in\frx}\inf_{y\in A}d(x,y),
\]
and this construction defines a left adjoint to the forgetful functor $(-)_0:\AP\to\MET$. Furthermore, note that $a_d(\doo{x},y)=d(x,y)$. In particular, for the metric space $[0,\infty]=([0,\infty],\mu)$ one obtains
\[
 a_\mu(\frx,y)=\sup_{A\in\frx}\inf_{x\in A}(y\trunminus x),
\]
and the approach space $([0,\infty],a_\mu)$ is not $\mU$-cocomplete. To see this, consider any ultrafilter $\frx\in U[0,\infty]$ which contains all sets
\[
 \{x\in[0,\infty]\mid u\le x<\infty\},
\]
$u\in[0,\infty]$ and $u<\infty$. Then $a_\mu(\frx,\infty)=\infty$ and $a_\mu(\frx,y)=0$ for all $y\in[0,\infty]$ with $y<\infty$, hence $a_\mu$ cannot be of the form $\mu(\alpha(-),-)$ for a map $\alpha:U[0,\infty]\to[0,\infty]$. However, for the metric space $[0,\infty]^\op=([0,\infty],\mu^\circ)$, the approach convergence structure $a_{\mu^\circ}$ is actually the structure induced by the metric compact Hausdorff space $([0,\infty],\mu^\circ,\xi)$ and therefore $([0,\infty],a_{\mu^\circ})$ is cocomplete.
\end{example}

$\mU$-cocomplete approach spaces are closely related to metric compact Hausdorff spaces respectively stably compact approach space, in both cases the approach structure $a$ on $X$ can be decomposed into a metric $a_0$ and a map $\alpha:UX\to U$, and one recovers $a$ as $a(\frx,x)=a_0(\alpha(\frx),x)$. In fact, every metric compact Hausdorff space is $\mU$-cocomplete, but the reverse implication is in general false since, for instance, the map $\alpha:UX\to X$ does not need to be an Eilenberg--Moore algebra structure on $X$ (i.e.\ a compact Hausdorff topology). Fortunately, this property of $\alpha$ was not needed in the proof of Lemma \ref{Lem:MapsMetCompSp}, and we conclude
\begin{lemma}
Let $(X,a)$ and $(Y,b)$ be $\mU$-cocomplete approach spaces and $f:X\to Y$ be a map. Then $f:(X,a)\to(Y,b)$ is an approach map if and only if $f:(X,a_0)\to(Y,b_0)$ is a metric map and, for all $\frx\in UX$, $\beta\cdot Uf(\frx)\le f\cdot\alpha(\frx)$.
\end{lemma}
\begin{remark}
Once again, everything told here has its topological counterpart. For instance, we call a topological space $X$ $\mU$-cocomplete whenever the monotone map $e_X:X_0\to(UX)_0$ has a left adjoint $\alpha:(UX)_0\to X_0$ in $\ORD$. Then, with $\le$ denoting the underlying order of $X$, an ultrafilter $\frx\in UX$ converges to $x\in X$ if and only if $\alpha(\frx)\le x$. Moreover, one also has an analog version of the lemma above.
\end{remark}
Recall from Subsection \ref{Subsect:App} that $(-)_p:\AP\to\TOP$ denotes the canonical forgetful functor from $\AP$ to $\TOP$, where $\frx\to x$ in $X_p$ if and only if $0=a(\frx,x)$ in the approach space $X=(X,a)$. If $X=(X,a)$ is also $\mU$-cocomplete with left adjoint $\alpha:(UX)_0\to X_0$, then, for any $\frx\in UX$ and $x\in X$,
\[
 \alpha(\frx)\le x\iff 0=a_0(\alpha(\frx),x)\iff 0=a(\frx,x)\iff\frx\to x.
\]
Here $\le$ denotes the underlying order of the underlying topology of $X$, which is the same as the underlying order of the underlying metric of $X$. Hence, $\alpha$ provides also a left adjoint to $e_X:X_{p0}\to U(X_p)_0$, and therefore the topological space $X_p$ is $\mU$-cocomplete as well. An important consequence of this fact is
\begin{proposition}\label{Prop:APcontrVsMetContr+Cont}
Let $X=(X,a)$ and $Y=(Y,b)$ be $\mU$-cocomplete approach spaces and $f:X\to Y$ be a map. Then $f:(X,a)\to(Y,b)$ is an approach map if and only if $f:(X,a_0)\to(Y,b_0)$ is a metric map and $f:X_p\to Y_p$ is continuous.
\end{proposition}
Finally, we also observe that $\mU$-cocomplete approach spaces are stable under standard constructions: both $X\oplus Y$ and $X\times Y$ are $\mU$-cocomplete, provided that $X=(X,a)$ and $Y=(Y,b)$ are so.

\subsection{Op-continuous lattices with an $[0,\infty]$-action}
\label{SubSect:OpContLatAct}

We call an approach space $X$ \emph{absolutely cocomplete} if the Yoneda embedding $\yoneda_X:X\to PX$ has a left adjoint in $\AP$. This is to say, $X$ is cocomplete and the metric left adjoint $\Sup_X$ of $\yoneda_X$ is actually an approach map $\Sup_X:PX\to X$. It is shown in \citep{Hof_Cocompl} that 
\begin{itemize}
\item the absolutely cocomplete approach spaces are precisely the injective ones, and that
\item  the category \[\COCTSAP\] of absolutely cocomplete approach T$_0$ spaces and supremum preserving ($=$ left adjoint) approach maps is monadic over $\AP$, $\MET$ and $\SET$. The construction $X\mapsto PX$ is the object part of the left adjoint $P:\AP\to\COCTSAP$ of the inclusion functor $\COCTSAP\to\AP$, and the maps $\yoneda_X:X\to PX$ define the unit $\yoneda$ of the induced monad $\mP=\pmonad$ on $\AP$. Composing this monad with the adjunction $(-)_d\dashv(-):\AP\leftrightarrows\SET$ gives the corresponding monad on $\SET$.
\end{itemize}
This resembles very much well-known properties of injective topological T$_0$ spaces, which are known to be the algebras for the filter monad on $\TOP$, $\ORD$ and $\SET$, hence, by Remark \ref{rem:FilterOfOpens}, are precisely the (accordingly defined) absolutely cocomplete topological T$_0$ spaces. Furthermore, all information about the topology of an injective T$_0$ space is contain in its underlying order, and the ordered sets occurring this way are the \emph{op-continuous lattices}, i.e. the duals of continuous lattices\footnote{Recall that our underlying order is dual to the specialisation order.}, as shown in \cite{Sco_ContLat} (see Subsection \ref{Subsect:ContLat}). In the sequel we will write $\OPCONTLAT$ to denote the category of op-continuous lattices and maps preserving all suprema and down-directed infima. Note that $\OPCONTLAT$ is equivalent to the category of absolutely cocomplete topological T$_0$ spaces and left adjoints in $\TOP$, and of course also to the category $\CONTLAT$ of continuous lattices and maps preserving up-directed suprema and all infima.

These analogies make us confident that absolutely cocomplete approach T$_0$ spaces provide an interesting metric counterpart to (op-)continuous lattices. In fact, in \citep{Hof_DualityDistSp} it is shown that the approach structure of such a space is determined by its underlying metric, hence we are talking essentially about metric spaces here. Moreover, every absolutely cocomplete approach space is exponentiable in $\AP$ and the \emph{full} subcategory of $\AP$ defined by these spaces is Cartesian closed. Theorem \ref{thm:CharAbsCocompII} below exposes now a tight connection with op-continuous lattices: the absolutely cocomplete approach T$_0$ spaces are precisely the op-continuous lattices equipped with an unitary and associative action of $[0,\infty]$ in the monoidal category $\OPCONTLAT$.

Every approach space $X=(X,a)$ induces approach maps 
\begin{align*}
 X\oplus[0,\infty]\xrightarrow{\,\formalball_X\,}PX, &&
 UX\xrightarrow{\,\yonedaT_X\,}PX, &&
 X^I\xrightarrow{\,\fammod_{X,I}\,}PX\qquad\text{(where $I$ is compact Hausdorff)}.
\end{align*}
Exactly as in Subsection \ref{Subsect:TensMet}, $\formalball_X:X\oplus[0,\infty]\to PX$ is the mate of the composite
\[
(UX)^\op\oplus X\oplus [0,\infty]\xrightarrow{\,a\oplus 1\,}
[0,\infty]\oplus [0,\infty]\xrightarrow{\,+\,}[0,\infty],
\]
and $\fammod_{X,I}:X^I\to PX$ is the mate of the composite
\[
 (UX)^\op\oplus X^I\to[0,\infty]^I\xrightarrow{\,\inf\,}[0,\infty],
\]
where the first component is the mate of the composite
\[
 (UX)^\op\oplus X^I\oplus I\xrightarrow{\,1\oplus\ev\,} (UX)^\op\oplus X
\xrightarrow{\,a\,}[0,\infty].
\]
Explicitely, for $\varphi\in X^I$ and $\frx\in UX$, $\fammod_{X,I}(\varphi)(\frx)=\inf_{i\in I}a(\frx,\varphi(i))$. A supremum of the ``down-set'' $\fammod_{X,I}(\varphi)\in PX$ is necessarily a supremum of the family $(\varphi(i))_{i\in I}$ in the underlying order of $X$.
If $X$ is cocomplete, one can compose the maps above with $\Sup_X$ and obtains metric maps
\begin{align}\label{Eq:ThreeMaps}
X_0\oplus[0,\infty]\xrightarrow{\,+\,}X_0, &&
(UX)_0\xrightarrow{\,\alpha\,} X_0, &&
(X^I)_0\xrightarrow{\,\bigvee\,}X_0\qquad\text{($I$ compact Hausdorff)},
\end{align}
which are even morphisms in $\AP$ provided that $X$ is absolutely cocomplete. In fact, one has
\begin{proposition}\label{Prop:CharAbsCocompI}
Let $X$ be an approach space. Then $X$ is absolutely cocomplete if and only if $X$ is cocomplete and the three maps \eqref{Eq:ThreeMaps} are approach maps.
\end{proposition}
\begin{proof}
The obtain the reverse implication, we have to show that the mapping
\[
 \Sup_X:PX\to X,\,\psi\mapsto\bigvee_{\frx\in UX}(\alpha(\frx)+\psi(\frx))
\]
is an approach map. We write $X_d$ for the discrete approach space with underlying set $X$, then $U(X_d)$ is just a compact Hausdorff space, namely the \v{C}ech-Stone compactification of the set $X$. By assumption, $\bigvee:X^{U(X_d)}\to X$ is an approach map, therefore it is enough to show that
\[
 U(X_d)\oplus PX\to X,\,(\frx,\psi)\mapsto\alpha(\frx)+\psi(\frx)
\]
belonges to $\AP$. Since the diagonal $\Delta:U(X_d)\to U(X_d)\oplus U(X_d)$ as well as the identity maps $U(X_d)\to UX$ and $U(X_d)\to (UX)^\op$ are in $\AP$, we can express the map above as the composite
\[
 U(X_d)\oplus PX\xrightarrow{\,\Delta\oplus 1\,}
UX\oplus (UX)^\op\oplus PX\xrightarrow{\,\alpha\oplus\ev\,}
X\oplus[0,\infty]\xrightarrow{\,+\,} X
\]
of approach maps.
\end{proof}

\begin{example}
The approach space $[0,\infty]$ is injective and hence absolutely cocomplete, but $[0,\infty]^\op$ is not injective. To see this, either observe that the map 
\[
f:\{0,\infty\}\to[0,\infty]^\op,\, 0\mapsto\infty,\infty\mapsto 0
\]
cannot be extended along the subspace inclusion $\{0,\infty\}\hrw[0,\infty]$, or that the mapping $(u,v)\mapsto u\trunminus v$ (which is the tensor of the metric space $[0,\infty]^\op$) is not an approach map of type $[0,\infty]^\op\oplus[0,\infty]\to[0,\infty]^\op$. Therefore $[0,\infty]^\op$ is not absolutely cocomplete, however, recall from Example \ref{ex:CocomplNotUCocompl} that $[0,\infty]^\op$ is cocomplete.
\end{example}

\begin{remark}
Similarly, a topological space $X$ is absolutely cocomplete if and only if $X$ is cocomplete and the latter two maps of \eqref{Eq:ThreeMaps} (accordingly defined) are continuous. 
\end{remark}

\begin{lemma}\label{Lem:XpowerIisUcocomplete}
Let $X$ be an approach space and $I$ be a compact Hausdorff space. If $X$ is cocomplete and $X_p$ is absolutely cocomplete, then $X^I$ is $\mU$-cocomplete.
\end{lemma}
\begin{proof}
We write $a:UX\times X\to[0,\infty]$ for the convergence structure of the approach space $X$, and $b:UI\times I\to\two$ for the convergence structure of the compact Hausdorff space $I$. In the both cases there are maps $\alpha:UX\to X$ and $\beta:UI\to I$ respectively so that $a(\frx,x)=a_0(\alpha(\frx),x)$ and $b(\fru,i)=\true$ if and only if $\beta(\fru)=i$, for all $\frx\in UX$, $x\in X$, $\fru\in UI$ and $i\in I$.
For every $\frp\in U(X^I)$ and $h\in X^I$,
\begin{align*}
 \fspstr{\frp}{h}
&=\sup\{a_0(\alpha(U\ev(\frw)),h(\beta(\fru))\mid \frw\in U(X^I\times I),\frw\mapsto\frp,(\frw\mapsto\fru)\}\\
&=\sup_{i\in I}\sup_{\substack{\frw\in U(X^I\times I)\\ U\pi_1(\frw)=\frp\\ \beta\cdot U\pi_2(\frw)=i}}a_0(\alpha(U\ev(\frw)),h(i))\\
&=\sup_{i\in I}a_0(\bigvee_{\substack{\frw\in U(X^I\times I)\\ U\pi_1(\frw)=\frp\\ \beta\cdot U\pi_2(\frw)=i}}\alpha(U\ev(\frw)),h(i))\\
&=\sup_{i\in I}a_0(\gamma(\frp)(i),h(i)),
\end{align*}
where we define
\[
 \gamma(\frp)(i)=\bigvee_{\substack{\frw\in U(X^I\times I)\\ U\pi_1(\frw)=\frp\\ \beta\cdot U\pi_2(\frw)=i}}\alpha(U\ev(\frw)).
\]
In order to conclude that $\gamma$ is a map of type $U(X^I)\to X^I$, we have to show that $\gamma(\frp)$ is a continuous map $\gamma(\frp):I\to X_p$, for every $\frp\in U(X^I)$. To this end, we note first that the supremum above can be rewritten as
\[
 \gamma(\frp)(i)=\bigvee_{\substack{\frw\in U(X^I\times I)\\ U\pi_1(\frw)=\frp}}\alpha(U\ev(\frw))\,\&\, b(U\pi_2(\frw),i).
\]
We put $Y=\{\frw\in U(X^I\times I)\mid U\pi_1(\frw)=\frp\}$ and consider $Y$ as a subspace of $U((X^I\times I)_d)$, that is, the \v{C}ech-Stone compactification of the set $X^I\times I$. Note that $Y$ is compact, and one has continuous maps
\begin{align*}
 Y\xrightarrow{\,U\!\ev\,}U(X_d), &&
 Y\xrightarrow{\,U\pi_2\,}U(I_d), &&
 U(X_d)\xrightarrow{\,\alpha\,}X_p, &&
 U(I_d)\times I\xrightarrow{\,b\,}\two.
\end{align*}
Therefore we can express the map $\gamma(\frp)$ as the composite
\[
 I\longrightarrow X_p^I\xrightarrow{\,\bigvee\,}X_p
\]
of continuous maps, where the first component is the mate of the composite
\[
 Y\times I\xrightarrow{\,\Delta\times 1\,}
 Y\times Y\times I\xrightarrow{\,U\!\ev\times U\pi_2\times 1\,}
 U(X_d)\times U(I_d)\times I\xrightarrow{\,\alpha\times b\,}
 X_p\times\two\xrightarrow{\,\&\,}X_p
\]
of continuous maps.
\end{proof}

\begin{proposition}\label{Prop:Cocompl+OPcont}
Let $X=(X,d)$ be a cocomplete metric space whose underlying ordered set $X_p$ is a op-continuous lattice. Then $(X,d,\alpha)$ is a metric compact Hausdorff space where $\alpha:UX\to X$ is defined by $\displaystyle{\frx\mapsto\bigwedge_{A\in\frx}\bigvee_{x\in X}x}$. 
\end{proposition}

\begin{proof}
Since $X_p$ is op-continuous, $X_p$ is even an ordered compact Hausdorff space with convergence $\alpha$. We have to show that $\alpha:U(X,d)\to(X,d)$ is a metric map. Recall from Lemma \ref{Lem:UXisTensored} that with $(X,d)$ also $U(X,d)$ is tensored, hence we can apply Proposition \ref{Prop:ContrTensMet}. Firstly, for $\frx\in UX$ and $u\in[0,\infty]$,
\[
 \alpha(\frx)+u=\left(\bigwedge_{A\in\frx}\bigvee_{x\in X}x\right)+u
\le\bigwedge_{A\in\frx}\bigvee_{x\in X}(x+u)=\alpha(\frx+u)
\]
since $-+u:X\to X$ preserves suprema. Secondly, let $\frx,\fry\in UX$ and assume
\[
 0=Ud(\frx,\fry)=\sup_{A\in\frx,B\in\fry}\inf_{x\in A,y\in B}d(x,y)
=\sup_{B\in\fry}\inf_{A\in\frx}\sup_{x\in A}\inf_{y\in B}d(x,y).
\]
For the last equality see \citep[Lemma 6.2]{SEAL_LaxAlg}, for instance. We wish to show that $\alpha(\frx)\le\alpha(\fry)$, that is,
\[
 \bigwedge_{A\in\frx}\bigvee_{x\in A}x\le \bigwedge_{B\in\fry}\bigvee_{y\in B}y,
\]
which is equivalent to $\displaystyle{\bigwedge_{A\in\frx}\bigvee_{x\in A}x\le\bigvee_{y\in B}y}$, for all $B\in\fry$. Let $B\in\fry$ and $\eps>0$. By hypothesis, there exist some $A\in\frx$ with $\sup_{x\in A}\inf_{y\in B}d(x,y)<\eps$, hence, for all $x\in A$, there exist some $y\in B$ with $d(x,y)<\eps$ and therefore $x+\eps\le y$. Consequently, for all $\eps>0$,
\[
 \left(\bigwedge_{A\in\frx}\bigvee_{x\in A}x\right)+\eps\le
 \bigwedge_{A\in\frx}\bigvee_{x\in A}(x+\eps)\le \bigvee_{y\in B}y;
\]
and therefore also $\displaystyle{\bigwedge_{A\in\frx}\bigvee_{x\in A}x\le\bigvee_{y\in B}y}$.
\end{proof}

\begin{theorem}
Let $X$ be a T$_0$ approach space. Then the following assertions are equivalent.
\begin{eqcond}
\item $X$ is absolutely cocomplete.
\item $X$ is cocomplete, $X_p$ is absolutely cocomplete and $+:X\oplus[0,\infty]\to X$ is an approach map.
\item $X$ is cocomplete, $X_p$ is absolutely cocomplete and $+:X_p\times[0,\infty]_p\to X_p$ is continuous.
\item $X$ is $\mU$-cocomplete, $X_0$ is cocomplete, $X_p$ is absolutely cocomplete, and, for all $x\in X$ and $u\in[0,\infty ]$, the map $-+u:X\to X$ preserves down-directed infima and the map $x+-:[0,\infty]\to X$ sends up-directed suprema to down-directed infima.
\end{eqcond}
\end{theorem}
\begin{proof}
Clearly, (i)$\Rw$(ii)$\Rw$(iii)$\Rw$(iv). Assume now (iv). According to Proposition \ref{Prop:CharAbsCocompI}, we have to show that the three maps \eqref{Eq:ThreeMaps} are approach maps. We write $a:UX\times X\to[0,\infty]$ for the convergence structure of the approach space $X$, by hypothesis, $a(\frx,x)=d(\alpha(\frx),x)$ where $d$ is the underlying metric and $\alpha(\frx)=\displaystyle{\bigwedge_{A\in\frx}\bigvee_{x\in X}x}$. By Lemma \ref{Prop:Cocompl+OPcont}, $(X,d,\alpha)$ is a metric compact Hausdorff space and therefore $\alpha:UX\to X$ is an approach map. Since the metric space $X_0$ is cocomplete, $+:X_0\oplus[0,\infty]\to X_0$ and $\bigvee:X_0^I\to X_0$ ($I$ any set) are metric maps. If $I$ is a compact Hausdorff space, $(X^I)_0$ is a subspace of $X_0^{I_d}$, therefore also $\bigvee:(X^I)_0\to X_0$ is a metric map. Furthermore, since $X_p$ is absolutely cocomplete, $\bigvee:(X^I)_p=(X_p)^I\to X_p$ is continuous (see Lemma \ref{Lem:PowerVsUnderlyingTop}). Since $X^I$ is $\mU$-cocomplete by Lemma \ref{Lem:XpowerIisUcocomplete}, $\bigvee:X^I\to X$ is actually an approach map by Proposition \ref{Prop:APcontrVsMetContr+Cont}. Similarly, $+:(X\oplus[0,\infty])_0\to X_0$ is a metric map since $(X\oplus[0,\infty])_0=X_0\oplus[0,\infty]$. Our hypothesis states that $+:(X\oplus[0,\infty])_p=X_p\times[0,\infty]_p\to X_p$ is continuous in each variable, and \citep[Proposition 2.6]{Sco_ContLat} tells us that it is indeed continuous. By applying Proposition \ref{Prop:APcontrVsMetContr+Cont} again we conclude that $+:X\oplus[0,\infty]\to X$ is an approach map. 
\end{proof}

Note that the approach structure of an absolutely cocomplete T$_0$ approach space can be recovered from its underlying metric since the convergence $\alpha:UX\to X$ is defined by the underlying lattice structure. In fact, the Theorem above shows that an absolutely cocomplete T$_0$ approach space is essentially the same thing as a separated cocomplete metric space $X=(X,d)$ whose underlying ordered set is an op-continuous lattice (see Proposition \ref{Prop:Cocompl+OPcont}) and where the action $+:X\times[0,\infty]\to X$ preserves down-directed infima (in both variables). In the final part of this paper we combine this with Theorem \ref{Thm:TensMetVsAction} where separated cocomplete metric spaces are described as sup-lattices $X$ equipped with an unitary and associative action $+:X\times[0,\infty]\to X$ on the set $X$ which preserves suprema in each variable, or, equivalently, $+:X\otimes[0,\infty]\to X$ is in $\SUP$.

For $X,Y,Z$ in $\OPCONTLAT$, a map $h:X\times Y\to Z$ is a bimorphism if it is a morphism of $\OPCONTLAT$ in each variable.

\begin{proposition}
The category $\OPCONTLAT$ admits a tensor product which represents bimorphisms. That is, for all $X,Y$ in $\OPCONTLAT$, the functor
\[
 \Bimorph(X\times Y,-):\OPCONTLAT\to\SET
\]
is representable by some object $X\otimes Y$ in $\OPCONTLAT$.
\end{proposition}
\begin{proof}
One easily verifies that $\Bimorph(X\times Y,-)$ preserves limits. We check the solution set condition of Freyd's Adjoint Functor Theorem (in the form of \citep[Section V.3, Theorem 3]{MacLane_WorkMath}). Take $\calS$ as any representing set of $\{L\in \OPCONTLAT\mid |L|\le|F(X\times Y)|\}$, where $F(X\times Y)$ denotes the set of all filters on the set $X\times Y$. Let $Z$ be an op-continuous lattice and  $\varphi:X\times Y\to Z$ be a bimorphism. We have to find some $L\in\calS$, a bimorphism $\varphi':X\times Y\to L$ and a morphism $m:L\to Z$ in $\OPCONTLAT$ with $m\cdot \varphi'=\varphi$. Since the map $e:X\times Y\to F(X\times Y)$ sending $(x,y)$ to its principal filter gives actually the reflection of $X\times Y$ to $\OPCONTLAT$, there exists some $f:F(X\times Y)\to Z$ in $\OPCONTLAT$ with $f\cdot e=\varphi$.
\[
\xymatrix{ & F(X\times Y)\ar@{-->>}[dr]^q\ar[dd]^f\\
&& L\ar@{>-->}[dl]^m\\
X\times Y\ar[uur]^e\ar[r]_\varphi\ar@{..>}[urr]^{\varphi'=q\cdot e\quad} & Z}
\]
Let $f=m\cdot q$ a (regular epi,mono)-factorisation of $f$ in $\OPCONTLAT$. Then $\varphi':=q\cdot e$ is a bimorphism as it is the corestriction of $\varphi$ to $L$, $m:L\to Z$ lies in $\OPCONTLAT$ and $L$ can be chosen in $\calS$.
\end{proof}
By unicity of the representing object, $1\otimes X\simeq X\simeq X\otimes 1$ and $(X\otimes Y)\otimes Z\simeq X\otimes(Y\otimes Z)$. Furthermore, with the order $\geqslant$, $[0,\infty]$ is actually a monoid in $\OPCONTLAT$ since $+:[0,\infty]\times[0,\infty]\to[0,\infty]$ is a bimorphism and therefore it is a morphism $+:[0,\infty]\otimes[0,\infty]\to[0,\infty]$ in $\OPCONTLAT$, and so is $1\to [0,\infty],\,\star\mapsto 0$. We write
\[
 \OPCONTLAT^{[0,\infty]}
\]
for the category whose objects are op-continuous lattices $X$ equipped with a unitary and associative action $+:X\otimes[0,\infty]\to X$ in $\OPCONTLAT$, and whose morphisms are those $\OPCONTLAT$-morphisms $f:X\to Y$ which satisfy $f(x+u)=f(x)+u$, for all $x\in X$ and $u\in[0,\infty]$.

Summing up, 
\begin{theorem}\label{thm:CharAbsCocompII}
$\COCTSAP$ is equivalent to $\OPCONTLAT^{[0,\infty]}$.
\end{theorem}
Here an absolutely cocomplete T$_0$ approach space $X=(X,a)$ is sent to its underlying ordered set where $x\le y\iff a(\doo{x},y)=0$ ($x,y\in X$) equipped with the tensor product of $X$, and an op-continuous lattice $X$ with action $+$ is sent to the approach space induced by the metric compact Hausdorff space $(X,d,\alpha)$ where $d(x,y)=\inf\{u\in[0,\infty]\mid x+u\le y\}$ and $\displaystyle{\alpha(\frx)=\bigwedge_{A\in\frx}\bigvee_{x\in A}x}$, for all $x,y\in X$ and $\frx\in UX$.

\begin{remark}
By the theorem above, the diagram
\[
\xymatrix{\COCTSAP\simeq\OPCONTLAT^{[0,\infty]}\ar[dd]_\dashv\ar[rr]^-{\perp} & &
 \OPCONTLAT\ar[ddll]^\rightthreetimes\ar@{.>}@/_4ex/[ll]_{-\otimes[0,\infty]}\\ \\
\SET\ar@{.>}@/^4ex/[uu]^P\ar@{.>}@/_4ex/[uurr]_F}
\]
of right adjoints commutes, and therefore the diagram of the (dotted) left adjoints does so too. Here $FX$ is the set of all filters on the set $X$, ordered by $\supseteq$, and $PX=[0,\infty]^{UX}$ where $UX$ is equipped with the Zariski topology. In other words, $PX\simeq FX\otimes[0,\infty]$, for every set $X$.
\end{remark}

\def\cprime{$'$}


\end{document}